\documentclass{amsart}
\usepackage[margin=1in]{geometry}
\usepackage{shuffle}
\usepackage{bm}

\usepackage{tikz}
\usepackage{mathdots}
\usepackage{array}
\usepackage{multirow}
\usepackage{hyperref}
\usepackage{float}
\usepackage{caption}
\usepackage{amscd}
\usepackage{tabularx}
\usepackage{extarrows}

\newcommand{\id}{\mathrm{id}}
\usepackage[all]{xy}
\usepackage{amsmath}
\usepackage{amsthm}
\usepackage{amssymb}
\usepackage{mathtools}
\usepackage{bbm}
\usepackage{soul}
\usepackage{color}

\usepackage{lmodern} 

\newtheorem{theorem}{Theorem}[section]

\newtheorem{proposition}[theorem]{Proposition}
\newtheorem{corollary}[theorem]{Corollary}
\newtheorem{lemma}[theorem]{Lemma}
\theoremstyle{definition}
\newtheorem{definition}[theorem]{Definition}
\newtheorem{example}[theorem]{Example}

\newcommand{\N}{\mathbb{N}}

\newcommand{\shifted}{\overrightarrow{\shuffle}}

\newcommand{\mix}{\ \rotatebox[origin=c]{180}{$\shuffle$}\ }
\newcommand{\smix}{\overrightarrow{\rotatebox[origin=c]{180}{$\shuffle$}}}

\usetikzlibrary{cd}

\newcommand{\pss}{(W/st)|_{\mathfrak{P}}}

\DeclareMathOperator{\pack}{pack}

\DeclareMathOperator{\Set}{Set}
\DeclareMathOperator{\Comp}{Comp}
\DeclareMathOperator{\std}{std}
\DeclareMathOperator{\shift}{shift}

\DeclareMathOperator{\St}{\mathsf{st}}

\DeclareMathOperator{\Sym}{Sym}
\DeclareMathOperator{\QSym}{QSym}
\DeclareMathOperator{\FQSym}{FQSym}
\DeclareMathOperator{\PQSym}{PQSym}
\DeclareMathOperator{\WQSym}{WQSym}
\DeclareMathOperator{\NSym}{NSym}
\DeclareMathOperator{\NCSym}{NCSym}
\DeclareMathOperator{\NCQSym}{NCQSym}
\DeclareMathOperator{\CQSym}{CQSym}

\DeclareMathOperator{\PF}{PF}
\DeclareMathOperator{\WIPF}{WIPF}
\DeclareMathOperator{\Sh}{Sh}
\DeclareMathOperator{\tcl}{\mathsf{tcl}}

\DeclareMathOperator{\Inv}{\mathsf{Inv}}
\DeclareMathOperator{\inv}{\mathsf{inv}}
\DeclareMathOperator{\maj}{\mathsf{maj}}
\DeclareMathOperator{\des}{\mathsf{des}}
\DeclareMathOperator{\asc}{\mathsf{asc}}
\DeclareMathOperator{\tie}{\mathsf{tie}}

\DeclareMathOperator{\Des}{\mathsf{Des}}
\DeclareMathOperator{\Asc}{\mathsf{Asc}}
\DeclareMathOperator{\Tie}{\mathsf{Tie}}
\DeclareMathOperator{\Cliff}{\mathsf{Cliff}}
\DeclareMathOperator{\cliff}{\mathsf{cliff}}
\DeclareMathOperator{\Peak}{\mathsf{Peak}}

\DeclareMathOperator{\lucky}{\mathsf{lucky}}
\DeclareMathOperator{\unlucky}{\mathsf{unlucky}}
\DeclareMathOperator{\Lucky}{\mathsf{Lucky}}
\DeclareMathOperator{\Unlucky}{\mathsf{Unlucky}}
\DeclareMathOperator{\disp}{\mathsf{disp}}
\DeclareMathOperator{\sumst}{\mathsf{sum}}
\DeclareMathOperator{\Disp}{\mathsf{Disp}}

\DeclareMathOperator{\out}{\mathsf{Out}}
\DeclareMathOperator{\lel}{\mathsf{lel}}
\DeclareMathOperator{\ones}{\mathsf{ones}}
\DeclareMathOperator{\isi}{\mathsf{isi}}

\DeclareMathOperator{\winc}{\mathsf{sort}\uparrow}
\DeclareMathOperator{\occ}{\mathsf{occ}}

\DeclareMathOperator{\bk}{\mathsf{bk}}
\DeclareMathOperator{\rank}{\mathsf{rank}}
\DeclareMathOperator{\Succ}{\mathsf{Succ}}
\DeclareMathOperator{\suc}{\mathsf{succ}}

\DeclareMathOperator{\Bk}{\mathsf{Bk}}

\DeclareMathOperator{\sg}{\mathsf{sg}}
\DeclareMathOperator{\Sg}{\mathsf{Sg}}

\DeclareMathOperator{\Op}{\mathsf{Op}}
\DeclareMathOperator{\Cl}{\mathsf{Cl}}
\DeclareMathOperator{\Tr}{\mathsf{Tr}}

 \usepackage[normalem]{ulem}

\title[Shuffle-compatibility for combinatorial statistics]{Shuffle-compatibility for combinatorial statistics on words, parking functions, and set partitions}

\author[S. Daugherty]{Spencer Daugherty}\address{Department of Mathematics, University of Colorado Boulder, Boulder, CO}\email{spencer.daugherty@colorado.edu}

\author[J. Liang]{Jinting Liang}\address{Department of Mathematics, University of British Columbia, Vancouver, BC}\email{liangj@math.ubc.ca}

\usepackage[backend=biber, maxbibnames=10, doi=false, url=false, isbn=false]{biblatex}
\begin{document}

\begin{abstract} We introduce notions of (weak) shuffle-compatibility for statistics on words, parking functions, and set partitions, generalizing Gessel and Zhuang's shuffle-compatibility for statistics on permutations. For parking functions and set partitions, we perform a systematic review of statistics that appear in the FindStat database (as well as the literature). We further define (shifted) shuffle algebras of (weakly) shuffle-compatible statistics on the equivalence classes induced by the statistics. These algebras relate closely to various combinatorial Hopf algebras such as $\QSym$, $\FQSym$, $\PQSym$, and $\NCSym$. These constructions yield new combinatorial interpretations of various Hopf algebra bases and, in some cases, new bases entirely.
   
\end{abstract}
\keywords{Shuffle-compatibility, Hopf algebra, combinatorial statistics, permutations, words, parking functions, set partitions,  FQSym, WQSym, PQSym, QSym, NCSym}

\maketitle

\section{Introduction}

In \cite{Ges18}, Gessel and Zhuang introduce the concept of shuffle-compatibility for statistics on permutations. This theory formalized an interesting phenomenon exhibited by certain permutation statistics explored most notably by Stanley's shuffling theorem. The authors also defined shuffle algebras for shuffle-compatible statistics, thus unifying the results of Stanley, Gessel, Stembridge, Aguiar-Bergeron-Nyman, and Petersen.  Many of the key results on shuffle-compatibility center around the descent sets of permutations and connect to the theory of P-partitions and the quasisymmetric functions.  Shuffle-compatibility has since been studied (and generalized) in \cite{bakerjarvis2019bijectiveproofsshufflecompatibility, carnevale2025colouredshufflecompatibilityhadamard, Grinberg_2018, Kantarc_O_uz_2022, liang2023cyclicshufflecompatibilitycyclicshuffle, yang2022conjectureconcerningshufflecompatiblepermutation}.

For $m,n \in \N = \{1,2, \ldots\}$ with $m<n$, let $[n] = \{1,2, \ldots,n\}$ and $[m,n] = \{m,m+1, \ldots, n-1, n\}$. Let a \emph{permutation} of length $n$ be a sequence $\pi = (\pi_1,\pi_2, \ldots, \pi_n)$ of distinct positive integers. 
We denote the set of permutations of length $n$ by $\mathfrak{P}_n$ and let $\mathfrak{P} = \cup_{n \geq 0} \mathfrak{P}_n$. The \emph{standardization} of a permutation $\pi$ with entries $\pi_{i_1} < \pi_{i_2} < \cdots < \pi_{i_n}$, denoted $\std(\pi)$, is obtained by replacing each entry $\pi_{i_j}$ with $j$. 
Thus, $\std(\pi)$ will be a permutation with elements $[n]$. 
We denote the set of permutations of $[n]$ by $\mathfrak{S}_n$ and let $\mathfrak{S} = \cup_{n \geq 0} \mathfrak{S}_n$.  The \emph{shuffle product} of two disjoint permutations $\tau$ and $\theta$, denoted $\tau \shuffle \theta$, is the set of permutations $\pi$ made up of the elements in $\tau$ and $\theta$ with $\tau$ and $\theta$ as distinct subwords. In other words, it is all possible ways of interleaving (or shuffling) the permutations $\tau$ and $\theta$.
A permutation statistic $\St$ is a function on permutations, often mapping permutations to integers or sets of integers. The $\St$-equivalence classe of $\pi \in \mathfrak{P}$ is given by $[\pi]^{\mathfrak{P}}_{\St} =\{\tau \in \mathfrak{P} : \St(\pi)=\St(\tau), |\pi|=|\tau|\}$, although we will drop the $\mathfrak{P}$ subscript when it is clear from context. Let $\mathfrak{P}/\St$ denote the set of all such equivalence classes.

\begin{definition}\cite{Ges18}
    A permutation statistic $\St$ is \emph{shuffle-compatible} if
    \begin{enumerate}
        \item for any disjoint permutations $\tau,\theta \in \mathfrak{P}$, the multiset $\{\{ \St(\pi) : \pi \in \tau \shuffle \theta \}\}$ depends only on $\St(\tau)$, $\St(\theta)$, $|\tau|$, and $|\theta|$, and
        \item for any permutations $\tau, \theta \in \mathfrak{P}$, if $\std(\tau)=\std(\theta)$, then $\St(\tau)=\St(\theta)$.
    \end{enumerate} The \emph{shuffle algebra} of $\St$, denoted $\mathcal{A}^{\mathfrak{P}}_{\St}$, is then defined as the algebra with basis elements $[\pi]^{\mathfrak{P}}_{\St} \in \mathfrak{P}/\St$ and the product $[\tau]^{\mathfrak{P}}_{\St}[\theta]^{\mathfrak{P}}_{\St} = \sum_{\pi \in \tau \shuffle \theta} [\pi]^{\mathfrak{P}}_{\St}$.
\end{definition}

The key example of a shuffle-compatible statistic is the \emph{descent set}, which is defined for $\pi \in \mathfrak{P}_n$ as \[\Des(\pi) = \{i : \pi_i > \pi_{i+1}, i \in [n-1] \}.\] The descent set is shuffle-compatible and, notably, the shuffle algebra of the descent set is isomorphic to the algebra of quasisymmetric functions, $\QSym$. As an example of this property, observe that $\Des((5))=\{\}$ and $\Des((2,6,4)) =\{2\}$. The shuffles of $(5)$ and $(2,6,4)$ are $(5,2,6,4), (2,5,6,4), (2,6,5,4)$ and $(2,6,4,5)$.  Thus, \[\{\{\Des(\pi) : \pi \in (5) \shuffle (2,6,4) \}\} = \{\{ \ \{1,3\}, \{3\}, \{2,3\}, \{2\} \ \}\}.\]
Similarly, we have that $\Des((1))=\{\}$ and $\Des((2,4,3)) = \{2\}$. The shuffles of $(1)$ and $(2,4,3)$ are $(1,2,4,3), (2,1,4,3), (2,4,1,3)$ and $(2,4,3,1)$. Thus, \[\{\{\Des(\pi) : \pi \in (1) \shuffle (2,4,3) \}\} = \{\{ \ \{3\}, \{1,3\}, \{2\}, \{2,3\} \ \}\}.\] The two multisets are equal because they correspond to pairs of permutations with the same lengths and descent sets, and the descent set is shuffle-compatible. Additionally, the shuffle algebra of the Descent set, $\mathcal{A}_{\Des}^{\mathfrak{P}}$, is isomorphic to the algebra of quasisymmetric functions. In fact, the product of the fundamental quasisymmetric functions is generally understood in terms of descent sets of shuffles of permutations. 

Gessel and Zhuang investigate statistics that depend only on the descent set and length of a permutation, which they call descent statistics. These include the descent number, the major index, the peak set, the peak number, the left peak set, the left peak number, and the number of up-down runs. They also consider shuffle-compatibility for tuples of these statistics and more. The authors conjectured that all shuffle-compatible statistics are descent statistics, although this was disproven in \cite{Kantarc_O_uz_2022} (though not with a statistic with notable combinatorial meaning).

Many of the proofs in \cite{Ges18} are algebraic in nature and use results from the literature on $q$-analogues of these statistics. For example, Stanley's shuffling theorem states that, for any disjoint $\tau \in \mathfrak{P}_m$, $\theta \in \mathfrak{P}_n$,  \[\sum_{\pi \in \tau \shuffle \theta} q^{\maj(\pi)} = q^{\maj(\tau) + \maj(\theta)} \binom{m+n}{m}_q,\] where $\maj$ is the major index (see Table \ref{tab:wordstats}) and $\binom{-}{-}_q$ is the $q$-binommial coefficient \cite{stanley1972ordered}.
This theorem implies that the major index is shuffle-compatible on permutations, given that the multiset of $\maj(\pi)$ values over all shuffle $\pi$ of $\tau$ and $\theta$ can be expressed in terms of only $\maj(\tau)$, $\maj(\theta)$, $m$, and $n$. The initial work on shuffle-compatibility is continued by Grinberg in \cite{Grinberg_2018} on more statistics and with additional focus on certain algebraic structures. Baker-Jarvis and Sagan study shuffle-compatibility on descent statistics as well in \cite{bakerjarvis2019bijectiveproofsshufflecompatibility}, but they employ bijective methods. 

In this paper, we generalize the combinatorial and algebraic constructions of shuffle-compatibility and shuffle algebras, along with their weaker forms, to words, parking functions, and set partitions. The generalization to words and parking functions is especially natural as both contain permutations as a subset, but allow for the study of various new statistics. Set partitions, too, admit a shuffle product and have many combinatorial statistics in the literature. All three sets, and their respective shuffling operations, relate closely to combinatorial Hopf algebras. In fact, the various (shifted) shuffle algebras we define are quotients of the Hopf algebra $\WQSym^*$, $\PQSym$, and $\NCSym^*$. These (shifted) shuffle algebras are also often isomorphic to Hopf algebras (or subalgebras of) $\FQSym$, $\QSym$, and $\CQSym$.

In addition to studying statistics that appear in the literature and newly defined statistics, we do a systematic review of the statistics listed for parking functions and set partitions in the FindStat database \cite{FindStat}. Similar reviews were carried out in \cite{adams2025cyclicsievingpermutations} and \cite{Elder_2023}, with a focus on homomesies and cyclic-sieving on permutations. We study 46 different statistics that are shuffle-compatible or weakly shuffle-compatible on words, parking functions, or set partitions. We list 120 statistics that fail to exhibit one or both properties in the appendix, along with an example to show how (weak) shuffle-compatibility is violated. 

In Section \ref{sec:prelim}, we cover background on the Hopf algebra structure of $\QSym$ and $\FQSym$. We then define weak shuffle-compatibility and shifted shuffle algebras for permutations. Using these constructions, we show that all shuffle algebras on permutations are quotient algebras of $\FQSym$ (see Theorem \ref{thm:weakFQSYM}). These constructions lay the groundwork for our generalizations in the following section. Note that while we do not study weak shuffle-compatibility on permutations directly, permutations are a subset of parking functions, and so any statistic that is weakly shuffle-compatible on parking functions is also weakly shuffle-compatible on permutations. 

In Section \ref{sec:words}, we consider the natural generalization of shuffle-compatibility and shuffle algebras to statistics on words. In this case, we show that shuffle algebras of shuffle-compatible statistics on words will be quotients of the combinatorial Hopf algebra $\WQSym^*$ (Proposition \ref{prop:alg_wqsym}). Many statistics that are shuffle-compatible on permutations are also shuffle-compatible on all words, such as the descent set, the descent number, the ascent set, the ascent number, and the major index. The algebraic results also translate exactly in these cases. The peak set (number), on the other hand, is not shuffle-compatible on all words, and so we define an analogue of peaks called the cliffs. The cliff set statistic is shuffle-compatible on all words, and its shuffle algebra is isomorphic to the shuffle algebra of the peak set statistic on permutations, the peak algebra (Corollary \ref{cor:cliff}). We also consider the tie set statistic, given the presence of repeated elements in words. The tie set shuffle algebra is isomorphic to $\QSym$, and using this fact, we formulate a basis of $\QSym$ that has a product defined combinatorially in terms of tie sets of words (Theorems \ref{thm:Tie} and \ref{thm:T_basis}).

In Section \ref{sec:parking}, we generalize weak shuffle-compatibility to the set of parking functions. We prove the weak shuffle-compatibility of various statistics from the literature and FindStat, including outcome, lucky car number, displacement, and inversion number (see Table \ref{tab:pfstats}). Generalizing our result from Section~\ref{sec:prelim}, we show that the shifted shuffle algebras of weakly shuffle-compatible statistics will be quotients of $\PQSym$, the Hopf algebra on parking functions (Proposition \ref{prop:alg_pqsym}). In many cases, these shifted shuffle algebras are isomorphic to the combinatorial Hopf algebras $\FQSym$, $\CQSym$, $\QSym$, and a subalgebra of $\QSym$. In particular, we relate the displacement sequence statistic to the shuffle basis of $\QSym$ (Theorem \ref{thm:Disp}) and the lucky car set statistic to binary shuffle bases of $\QSym$ (Theorem \ref{thm:Lucky}). The inversion set and outcome statistics both yield shifted shuffle algebras isomorphic to $\FQSym$ by simple maps to the $\mathbf{F}$-basis (Theorems \ref{thm:Inv} and \ref{thm:outcome}).

In Section \ref{sec:setpartitions}, we generalize shuffle-compatibility to set partitions using a shuffle operation inspired by the product in the Hopf algebra dual to the symmetric functions in noncommuting variables, $\NCSym^*$. This shuffle, which we call the arc-shuffle, amounts to shuffling the arc-diagrams of two set partitions while their labels remain fixed in place. Weak shuffle-compatibility for set partitions is defined similarly with a shifted version of the arc shuffle (this is the operation corresponding exactly to the product in $\NCSym^*$). We prove that the (shifted) shuffle algebras for (weakly) shuffle-compatible set partition statistics are quotients of $\NCSym^*$ (Propositions \ref{prop:setparquo} and  \ref{prop:alg_ncsym}). The weakly shuffle-compatible statistics we study for set partitions include the succession set and number, the set of singletons, and the sets of closers, openers, and transients. The shuffle-compatible statistics we study for set partitions include block number, rank, number of singletons, and numbers of occurrences of certain patterns (see Table \ref{tab:setP}).  Interestingly, the shifted shuffle algebra for the succession set is isomorphic to $\QSym$ (Theorem \ref{thm:Succ_WSC}), specifically mapping to the tie basis formulated in Theorem \ref{thm:T_basis}. We see a similar correspondence between the maximal block size statistic on set partitions and the maximal displacement of a single car statistic on parking functions (Corollary \ref{cor:maxBk}) as well as the cardinality of the block containing $1$ statistic on set partitions and the leading elements statistic on parking functions (Proposition \ref{prop:bk1}). The shifted shuffle algebras for the sets of transients, closers,  and openers are isomorphic to $\QSym$ (or a subalgebra of $\QSym$ in the first case), specifically mapping to the binary shuffle basis (Propositions \ref{prop:tr}, \ref{prop:opener}, and \ref{prop:cl}), while the shifted shuffle algebra for the singleton set corresponds to a different subalgebra of the shuffle algebra on binary words (Theorem \ref{thm:sg_set}). The shuffle algebra of the block sizes statistic proves to be isomorphic to $\Sym$, the algebra of symmetric functions (Theorem \ref{thm:bk}).
Finally, we close the paper with an appendix containing tables of statistics that are not (weakly) shuffle-compatible with a focus on parking functions and set partitions. 

\subsection*{Acknowledgements} The authors would like to thank Nat Thiem, Jessica Striker, and Stephanie van Willigenburg for helpful comments. The second author is supported in part by the Natural Sciences and Engineering Research Council of Canada.

\section{Preliminaries}\label{sec:prelim}

Some of the richest algebraic structures in combinatorics are Hopf algebras. We refer the reader to \cite{grinberg, Luoto2013} for background on Hopf algebras in general, and provide here only the relevant background needed. Note that we take $\mathbb{Q}$ as our base field for all algebras in this paper. 

The \emph{Hopf algebra of quasisymmetric functions}, denoted, $\QSym$ is defined as the subalgebra of $\mathbb{Q}[[x_1, x_2, \ldots]]$ such that for every tuple $(a_1, a_2, \ldots, a_k)$ of positive integers, any two monomials \[x_{i_1}^{a_1} x_{i_2}^{a_2} \cdots x_{i_k}^{a^k} \text{\quad and \quad }x_{j_1}^{a_1} x_{j_2}^{a_2} \cdots x_{j_k}^{a_k}\] with $i_1 < i_2 < \cdots < i_k$ and $j_1 < j_2 < \cdots < j_k$ have equal coefficients. 
The algebra of quasisymmetric functions is graded, $\QSym = \oplus_{n \geq 0} \QSym_n$ where $\QSym_n$ is the span of all quasisymmetric functions with homogeneous degree $n$. 
The dimension of $\QSym_n$ is equal to the number of integer compositions of $n$. An integer compositon of $n$ is a sequence of positive integers $\alpha = (\alpha_1, \alpha_2, \ldots, \alpha_k)$ such that $\sum_{i=1}^k \alpha_i = n$, and we write $\alpha \vDash n$.  
There is a bijection between integer compositions of $n$ and subsets of $[n-1]$. Specifically, for a composition $\alpha = (\alpha_1, \alpha_2, \ldots, \alpha_k)$, let $\Set(\alpha) = \{\alpha_1, \alpha_1 + \alpha_2, \ldots, \alpha_1 + \alpha_2 +\cdots + \alpha_{k-1}\}$. 
On the other hand, for a set $S = \{s_1 < s_2 < \cdots < s_{j}\} \subseteq [n-1]$, we have $\Comp(S) = (s_1, s_2-s_1, \ldots, s_{k} -s_{k-1}, n-s_{k})$.

The fundamental basis of $\QSym$ is given by $\{F_{\alpha}\}_{\alpha \vDash n}$ for the graded component $\QSym_n$. Let $\pi_{\alpha}$ be a permutation with $\Comp(\Des(\pi_{\alpha})) = \alpha$. The product of the fundamental basis elements is given by 
\[F_{\alpha}F_{\beta} = \sum_{\sigma \in \pi_{\alpha} \shuffle \pi_{\beta}} F_{\Comp(\Des(\sigma))}.\] It is because of the shuffle-compatibility of the descent set that this product is well-defined regardless of the choice of $\pi_{\alpha}$ and $\pi_{\beta}$. Moreover, the canonical isomorphism between $\mathcal{A}_{\St}^{\mathfrak{P}}$ and $\QSym$ is given by $[\pi]_{\Des} \mapsto F_{\Comp(\Des(\pi))}$. The quasisymmetric functions are of critical importance in algebraic combinatorics, with connections and applications to enumerative combinatorics, Macdonald polynomials, Coxeter groups, Hecke algebras, Kazhdan-Lusztig polynomials, and Schubert polynomials \cite{mason2018recenttrendsquasisymmetricfunctions}.

The Hopf algebra most relevant to shuffle-compatibility on permutations is the \emph{Malvenuto-Reutenaur Hopf algebra on permutations}, or the algebra of free quasisymmetric functions, $\FQSym$. 
This algebra is graded $\FQSym = \oplus_{n\geq0} \FQSym_n$ where $\FQSym_n$ has dimension $n!$. The $\mathbf{F}$-basis of $\FQSym$ is given by $\{\mathbf{F}_{\pi}\}_{\pi \in \mathfrak{S}_n}$ for the graded component $\FQSym_n$. 
For two permutations $\tau \in \mathfrak{S}_n$, $ \theta \in \mathfrak{S}_m$, let the shifted shuffle be defined as $\tau \shifted \theta = \tau \shuffle \theta[n]$ where $\theta[n] = (\theta_1 +n, \theta_2 +n, \ldots, \theta_m +n)$. 
The product on $\mathbf{F}$-basis elements is given by \[\mathbf{F}_{\tau}\mathbf{F}_{\theta} = \sum_{\pi \in \tau \shifted \theta} \mathbf{F}_{\pi}.\] 
This shifted shuffle product is used to define what we call \emph{weak shuffle-compatibility} for permutation statistics. This is a slightly more restrictive definition of the weak form of shuffle-compatibility mentioned (but not studied) in \cite{Ges18}. This phenomenon is studied (but not named) for peaks, valleys, double descents, and double ascents in \cite{vong2013algebraic} by Vong in the context of Laguerre histories.

\begin{definition}\label{def:weakperm}
    A permutation statistic $\St$ is \emph{weakly shuffle-compatible} if for any permutations $\tau, \theta \in \mathfrak{S}$, the multiset $\{\{ \St(\pi) : \pi \in \tau \shifted \theta \}\}$ depends only on $\St(\tau)$, $\St(\theta)$, $|\tau|$, and $|\theta|$.  
\end{definition}

For example, the number of inversions statistic (see Table \ref{tab:pfstats}) is weakly shuffle-compatible but not shuffle-compatible \cite{Ges18}. Any shuffle-compatible statistic is also weakly shuffle-compatible, and the algebraic constructions from weak shuffle-compatibility are, in fact, equivalent to those of shuffle-compatibility, as we show here. We denote the $\St$-equivalence class of $\pi \in \mathfrak{S}$ by $[\pi]^{\mathfrak{S}}_{\St} = \{\tau \in \mathfrak{S} : \St(\tau) = \St(\pi), |\tau|=|\pi|\}$, although we will drop the $\mathfrak{S}$ superscript when it is clear from context. Let $\mathfrak{S}/\St$ denote the set of all $\St$-equivalence classes on $\mathfrak{S}$.

\begin{definition}
Let $\St$ be a weakly shuffle-compatible statistic on permutations. The \emph{shifted shuffle algebra} of $\St$, denoted $\mathcal{A}^{\mathfrak{S}}_{\St}$, is defined as the algebra with basis elements $[\pi]_{\St}^{\mathfrak{S}} \in \mathfrak{S}/\St$ and the product \[[\tau]^{\mathfrak{S}}_{\St}[\theta]^{\mathfrak{S}}_{\St} = \sum_{\pi \in \tau \shifted \theta} [\pi]_{\St}^{\mathfrak{S}}.\]
\end{definition}

\begin{lemma}\label{lem:shift} If $\St$ is shuffle-compatible on permutations, then the shuffle algebra $\mathcal{A}_{\St}^{\mathfrak{P}}$ is isomorphic to the shifted shuffle algebra $\mathcal{A}_{\St}^{\mathfrak{S}}$.
\end{lemma}

\begin{proof}
    Consider the map $\psi: \mathcal{A}^{\mathfrak{P}}_{\St} \rightarrow \mathcal{A}^{\mathfrak{S}}_{\St}$ defined by $[\pi]^{\mathfrak{P}}_{\St} \mapsto [\std(\pi)]^{\mathfrak{S}}_{\St}$.  Because $\St$ is shuffle-compatible, $\St(\pi) = \St(\std(\pi))$ for any $\pi \in \mathfrak{P}$. As a result, $[\std(\pi)]^{\mathfrak{S}}_{\St}$ is exactly the nonempty subset of $[\pi]^{\mathfrak{P}}_{\St}$ made up of permutations in $\mathfrak{S}$. Thus, $\psi$ is a bijection. Observe that \[\psi([\tau]_{\St}^{\mathfrak{P}}[\theta]^{\mathfrak{P}}_{\St}) = \psi(\sum_{\pi \in \tau \shuffle \theta} [\pi]^{\mathfrak{P}}_{\St}) = \sum_{\pi \in \tau \shuffle \theta} [\std(\pi)]^{\mathfrak{S}}_{\St}.\] On the other hand, \[\psi([\tau]^{\mathfrak{P}}_{\St})\psi([\theta]^{\mathfrak{P}}_{\St}) = [\std(\tau)]^{\mathfrak{S}}_{\St}[\std(\theta)]_{\St}^{\mathfrak{S}} = \sum_{\rho \in \std(\tau) \shifted \std(\theta)} [\rho]^{\mathfrak{S}}_{\St}.\] These two expressions are equivalent given that $\std(\tau) \shifted \std(\theta) = \std(\tau) \shuffle \std(\theta)[|\tau|]$ and the shuffle-compatibility of $\St$ implies that $\{\{ \St(\pi) : \pi \in \tau \shuffle \theta \}\} = \{\{\St(\rho) : \rho \in \std(\tau) \shuffle \std(\theta)[|\tau|] \}\}$. Therefore, $\psi$ is an isomorphism.
\end{proof}

This construction allows us to prove a new result relating to shuffle algebras and their relationship to the Malvenuto-Reutenauer algebra Hopf algebra on permutations. 

\begin{theorem}\label{thm:weakFQSYM}
Let $\St$ be a statistic on permutations.
\begin{enumerate}
   \item  If $\St$ is a weakly shuffle-compatible statistic, the shifted shuffle algebra $\mathcal{A}_{\St}^{\mathfrak{S}}$ is a quotient of $\FQSym$. 
   \item If $\St$ is a shuffle-compatible statistic, the shuffle algebra $\mathcal{A}^{\mathfrak{P}}_{\St}$ is a quotient of $\FQSym$.
\end{enumerate}
\end{theorem}

\begin{proof}
    Let $\St$ be a weakly shuffle-compatible statistic. Consider the map $\varsigma: \FQSym \rightarrow \mathcal{A}^{\mathfrak{S}}_{\St}$. defined by $\mathbf{F}_{\pi} \rightarrow [\pi]_{\St}^{\mathfrak{S}}$. The map is clearly surjective. Additionally, \[\varsigma(\mathbf{F}_{\tau}\mathbf{F}_{\theta}) = \varsigma(\sum_{\pi \in \tau \shifted \theta} \mathbf{F}_{\pi})= \sum_{\pi \in \tau \shifted \theta} [\pi]^{\mathfrak{S}}_{\St} = [\tau]^{\mathfrak{S}}_{\St}[\theta]^{\mathfrak{S}}_{\St}=\varsigma(\mathbf{F}_{\tau})\varsigma(\mathbf{F}_{\theta}).\] Since $\varsigma$ is a surjective morphism, we have that $\mathcal{A}^{\mathfrak{S}}_{\St}$ is isomorphic to $\FQSym/\ker(\varsigma)$. This proves part (1) of the theorem. If $\St$ is shuffle-compatible, part (2) of the theorem follows from part (1) using Lemma \ref{lem:shift}.
\end{proof}

\section{Words}\label{sec:words}

We define a \emph{word} of length $n$ to be a sequence of positive integers $\alpha = (\alpha_1, \ldots, \alpha_{n})$. Let $W_n$ denote the set of words of length $n$ and let $W = \cup_{n \geq 0}W_n$. Given a word $\alpha$ with unique letters $a_1 < a_2 < \cdots <a_k$, let $\pack(\alpha)$ be the word obtained by replacing each $a_i$ with $i$. We say a word is \emph{packed} if $\pack(\alpha)=\alpha$, or in other words, if the letters used are exactly the elements of some set $\{1,2, \ldots, k\}$.

 We now recall various statistics on words.  Many of these statistics function similarly to corresponding statistics on permutations, except that it is now possible to have ties in addition to descents and ascents. Note that ties are sometimes referred to as plateaux in the literature \cite{Mona}.

\begin{table}[h]
    \centering
    \begin{tabular}{|c|c|} \hline
        \textbf{Name} & \textbf{Definition} \\ \hline
        Descent set & $\Des(\alpha) = \{ i : \alpha_i > \alpha_{i+1}\}$   \\ \hline
        Descent number & $\des(\alpha) = |\Des(\alpha)|$  \\ \hline
         Major index & $\maj(\alpha)=\sum_{\alpha_i>\alpha_{i+1}}i=\sum_{i\in\Des(\alpha)}i$   \\ \hline
        Ascent set & $\Asc(\alpha) = \{i : \alpha_i < \alpha_{i+1}\}$  \\ \hline
        Ascent number & $\asc(\alpha) = |\Asc(\alpha)|$  \\ \hline
         Cliff set & $\Cliff(\alpha) = \{ i : \alpha_{i-1} \leq \alpha_i > \alpha_{i+1} \}$ \\ \hline
        Cliff number & $\cliff(\alpha)=|\Cliff(\alpha)|$ \\ \hline
        Tie set & $\Tie(\alpha) = \{ i : \alpha_i = \alpha_{i+1} \}$  \\ \hline
        Tie number & $ \tie(\alpha) = |\Tie(\alpha)|$  \\ \hline
    \end{tabular}
    \caption{Statistics on words \cite{cruz2024some, 
    Ges18}}
    \label{tab:wordstats}
\end{table}

The Hopf algebra $\WQSym$ of \emph{word quasisymmetric functions}, also known as $\NCQSym$ \cite{BERGERON_2009}, is the algebra of noncommutative polynomial invariants of Hivert's quasi-symmetrizing action \cite{hivert199}. It also relates to algebras introduced by Chapoton involving associahedra and hypercubes \cite{Chapoton2000}. Its dual, $\WQSym^*$, is isomorphic to the Solomon-Tits algebra and is a subalgebra of the parking function quasisymmetric functions $\PQSym$ \cite{BERGERON_2009, novelli2006polynomialrealizationstrialgebras}. The algebra $\WQSym^*$ has a basis made up of elements $M^*_{\alpha}$ for every packed word $\alpha$. 
This basis has the product rule \[M^*_{\alpha}  M^*_{\beta} = \sum_{\gamma \in \alpha \shifted  \beta} M^*_{\gamma}.\]

We do not explicitly consider weak shuffle-compatibility on packed words in this paper, but packed words are a subset of parking functions, which we consider in Section \ref{sec:parking}. Any statistic that is weakly shuffle-compatible on parking functions will also be weakly shuffle-compatible on packed words, and parking functions provide a wider breadth of statistics to study.

\subsection{Shuffle-compatibility on words} We extend shuffle-compatibility to words using the natural extension of the shuffle product to words.
The \emph{shuffle product} of two words $\alpha$ and $\beta$, denoted $\alpha \shuffle \beta$, is the multiset of words $\gamma$ made up of the elements in $\alpha$ and $\beta$ with $\alpha$ and $\beta$ as distinct subwords. In other words, it is all possible ways of interleaving (or shuffling) the words $\alpha$ and $\beta$. 

\begin{example} 
The shuffle of words $(1,1)$ and $(2,3)$ is given by \[ \{\{ (1,1,2,3), (1,2,1,3), (2,1,1,3), (1, 2, 3, 1), (2, 1, 3, 1), (2, 3, 1, 1)\}\}.\]
    The shuffle product of words $(1,1)$ and $(1,2)$ is given by \[ \{\{ (1,1,1,2), (1,1,1,2), (1,1,1,2), (1, 1, 2, 1), (1, 1, 2, 1), (1, 2, 1, 1)\}\}.\]
\end{example}

\begin{definition}
    A statistic $\St$ is \emph{shuffle-compatible on words} if 
    \begin{enumerate}
    \item the multiset $\{\{ \St(\gamma) : \gamma \in \alpha \shuffle \beta \}\}$ is determined by $|\alpha|$, $|\beta|$, $\St(\alpha)$, and $\St(\beta)$ for any two disjoint words $\alpha, \beta \in W$, and
    \item if $\pack(\alpha)=\pack(\beta)$, then $\St(\alpha)=\St(\beta)$ for all words $\alpha, \beta \in W$.
    \end{enumerate}
\end{definition}

Equivalently, $\St$ is shuffle-compatible on words if \[\{\{\St(\gamma) : \gamma \in \alpha \shuffle \beta\}\}=\{\{\St(\gamma) : \gamma \in \kappa \shuffle \upsilon\}\}\] for any two pairs of disjoint words $\{\alpha, \beta\}$ and $\{\kappa, \upsilon\}$ where  $|\alpha|=|\kappa|$, $|\beta|=|\upsilon|$, $\St(\alpha) = \St(\kappa)$ and $\St(\beta)=\St(\upsilon)$. We choose to consider only shuffles of disjoint words, as shuffle-compatibility of permutations takes into account only pairs of disjoint permutations. Considering non-disjoint pairs of words would cause most of our statistics to fail to meet the shuffle-compatibility criteria.
    
    Let $[\alpha]_{\St}^W = \{\beta \in W : \St(\alpha)=\St(\beta), |\alpha|=|\beta|\}$ be the $\St$-equivalence class of  $\alpha \in W$, where we may drop the $W$ superscript when it is clear from context. Let $W/\St$ denote the set of all such equivalence classes.

\begin{definition}
 The $\St$ shuffle algebra, denoted $\mathcal{A}^{W}_{\St}$, is defined as the algebra with basis elements $[\alpha]_{\St}^W \in W/\St$ and the product \[ [\alpha]^W_{\St} [\beta]^W_{\St} = \sum_{ \gamma \in \alpha \shuffle \beta} [\gamma]^W_{\St}.\]
\end{definition}

Due to our definition of a statistic on words and use of the shuffle product, all our shuffle algebras will have the following property.

\begin{proposition}\label{prop:alg_wqsym}
    If $\St$ is shuffle-compatible on words, then $\mathcal{A}^{W}_{\St}$ is a quotient algebra of $\WQSym^*$.
\end{proposition}

\begin{proof} Let $\St$ be a shuffle-compatible statistic on words. Consider the map $\varsigma: \WQSym^* \rightarrow \mathcal{A}^W_{\St}$ given by $M^*_{\alpha} \mapsto [\alpha]_{\St}$. Because $\pack(\alpha)= \pack(\beta)$ implies $\St(\alpha)=\St(\beta)$, each $\St$-equivalence class on words contains a packed word, so the map is surjective. Additionally, using the fact that $\beta[n] \in [\beta]_{\St}$, \[\varsigma(M^*_{\alpha}M^*_{\beta}) = \varsigma(\sum_{\gamma \in \alpha \shifted \beta} M^*_{\gamma}) = \sum_{\gamma \in \alpha \shifted \beta } [\gamma]_{\St} = \sum_{\gamma \in \alpha \shuffle \beta[n]} [\gamma]_{\St} = [\alpha]_{\St}[\beta]_{\St} = \varsigma(M^*_{\alpha})\varsigma(M^*_{\beta}).\] Thus, $\varsigma$ is a surjective morphism, so $\mathcal{A}_{\St}^{W}$ is isomorphic to $\WQSym^*/\ker(\varsigma)$.
\end{proof}

We also have the following relationship between shuffle algebras from words and shuffle algebras from permutations.

\begin{proposition}\label{prop:perm_subalg}
    Let $\pss$ be the subset of equivalence classes in $W/\St$ that contain some permutation.  Then, $\mathcal{A}^{\mathfrak{P}}_{\St}$ is isomorphic to the subalgebra of $\mathcal{A}^{W}_{\St}$ spanned by $\pss$.
\end{proposition}

\begin{proof}
Define a map $f:\mathcal{A}^{\mathfrak{P}}_{\St} \rightarrow \mathcal{A}^{W}_{\St} $ where $f([\pi]^{\mathfrak{P}}_{\St}) = [\pi]^{W}_{\St}$ for $\pi \in \mathfrak{P}$. The image of the map $f$ is then given exactly by the span of $\pss$ in $\mathcal{A}_{\St}^W$. If $[\alpha]_{\St}^{W}$ and $[\beta]_{\St}^{W}$ are in the image, then they both contain permutations, say $\alpha'$ and $\beta'$. 
Note that we select these to be disjoint by shifting if necessary.
We can express their product as \[ [\alpha]_{\St}^{W}[\beta]_{\St}^{W} = \sum_{\gamma \in \alpha \shuffle \beta} [\gamma]^W_{\St} =  \sum_{\kappa \in \alpha' \shuffle \beta'} [\kappa]^{W}_{\St}.\] 
Since each $\kappa \in \alpha' \shuffle \beta'$ will also be a permutation, each $[\kappa]^{W}_{\St}$ is the image of some $[\kappa]^{\mathfrak{P}}_{\St}$ under $f$. Thus, the image of $f$ is closed under multiplication and is a subalgebra of $\mathcal{A}_{\St}^W$.

This map is injective because no two permutations in different equivalence classes of $\St$ over $\mathfrak{P}$ will ever be in the same equivalence class of $\St$ over $W$. Thus, it remains to check that $f$ is compatible with the product operation. Given $[\tau]_{\St}^{\mathfrak{P}}$, $[\theta]_{\St}^{\mathfrak{P}}$, observe 
\[f([\tau]_{\St}^{\mathfrak{P}}[\theta]_{\St}^{\mathfrak{P}}) = f(\sum_{\pi \in \tau \shuffle \theta} [\pi]^{\mathfrak{P}}_{\St}) = \sum_{\pi \in \tau \shuffle \theta} [\pi]^{W}_{\St} = [\tau]_{\St}^{W} [\theta]_{\St}^{W} = f([\tau]_{\St}^{\mathfrak{P}} f([\theta]_{\St}^{\mathfrak{P}}).\]
Thus $f$ is an isomorphism from $\mathcal{A}^{\mathfrak{P}}_{\St}$ to the subalgebra of $\mathcal{A}^{W}_{\St}$ spanned by $\pss$.
\end{proof}

We now consider statistics on words, beginning with the descent set statistic. The following result generalizes \cite[Lemma 2.8]{Ges18} to disjoint words. The proof follows the original closely. Given two weak compositions with the same number of parts $J = (J_1, \ldots,J_k)$ and $K = (K_1, \ldots, K_k)$, let $J+K = (J_1 + K_1, J_2 + K_2, \ldots, J_k + K_k)$.
If $\alpha$ is a word with descent set $\Des(\alpha)$, then let $\Comp(\alpha) = \Comp(\Des(\alpha))$.

\begin{lemma}\label{lem:des_lemma}
    Let $\alpha$ and $\beta$ be disjoint words with $|\alpha|=n$ and $|\beta|=m$. Let $A \subseteq [m+n-1]$ and let $L = \Comp(A)$. Then the number of shuffles of $\alpha$ and $\beta$ with descent set contained in $A$ is equal to the number of weak compositions $J \vDash n$ and $K \vDash m$ such that $J$ is a refinement of $\Comp(\alpha)$, $K$ is a refinement of $\Comp(\beta)$, $J$ and $K$ have the same number of parts as $L$, and $J+K=L$.
\end{lemma}

\begin{proof}
    Let $L$ have $k$ parts and let $J$ and $K$ satisfy the conditions above. Create a barred word from $\alpha$ by inserting bars after the $i^{\text{th}}$ entry for each $i \in \Set(J)$, and do the same for $\beta$ using $\Set(K)$. Thus $\alpha$ and $\beta$ will both be divided into $k$ blocks in which the letters must be weakly increasing. 
    To associate these barred words with a shuffle of $\alpha$ and $\beta$, do the following procedure. For each $1 \leq i \leq k$, shuffle the $i^{\text{th}}$ blocks of $\pi$ and $\beta$ together so that the entries are in weakly increasing order. Since the entries of $\alpha$ and $\beta$ are disjoint, there is only one way to do this for each block. Concatenating each of these new combined blocks in order will yield a shuffle of $\alpha$ and $\beta$ with a descent set contained in $A$.
    We show this is a bijection by giving the inverse map. Let $\gamma$ be a shuffle of $\alpha$ and $\beta$ with $\Des(\gamma) \subseteq A$ with $|A| = k-1$. Insert bars after the $i^{\text{th}}$ element of $\gamma$ for each $i \in A$, which will include the descents of $\gamma$, to create a composition $L$.  This induces a way of inserting $k-1$ bars into $\alpha$ and $\beta$ at the same relative positions of each to create compositions $J$ and $K$ respectively, both with exactly $k$ parts. Note that this method also creates $J$ and $K$ exactly so that $J+K=L$. Further, note that for any two elements that create a descent in $\alpha$ or $\beta$, there will have to be a descent somewhere between those two elements in $\gamma$.  As a result, we will place a bar between those two elements in $\gamma$ and also in $\alpha$ or $\beta$. This means that $J$ and $K$ will be refinements of $\Comp(\alpha)$ and $\Comp(\beta)$, respectively.
\end{proof}

By the inclusion-exclusion principle, the previous lemma implies that the number of shuffles of $\alpha$ and $\beta$ with a descent set equal to $A$ depends only on $\Des(\alpha)$, $\Des(\beta)$, $|\alpha|$, and $|\beta|$. It follows that the descent set statistic $\Des$ is shuffle-compatible on words.

\begin{theorem}\label{thm:SCofDes}(Shuffle-compatibility of the descent set)
\begin{enumerate}
    \item The descent set statistic $\Des$ is shuffle-compatible on words.
    \item The linear map on $\mathcal{A}^{W}_{\Des}$ defined by \[[\alpha]_{\Des} \rightarrow F_{\Comp(\Des(\alpha))},\] is an algebra isomorphism between $\mathcal{A}^{W}_{\Des}$ and $\QSym$. 
\end{enumerate}
    
\end{theorem}

\begin{proof} Note that each $\Des$-equivalence class contains at least one permutation. Thus, following Proposition~\ref{prop:perm_subalg}, we have that $\mathcal{A}_{\Des}^{W}$ is isomorphic to $\mathcal{A}^{\mathfrak{P}}_{\Des}$ via the map $[\alpha]^{W}_{\Des} \rightarrow [\pi]^{\mathfrak{P}}_{\Des}$ where $[\pi]^{\mathfrak{P}}_{\Des} \subseteq [\alpha]^{W}_{\Des}$. Composing this with the isomorphism $[\pi]^{\mathfrak{P}}_{\Des} \mapsto F_{\Comp(\Des(\pi))}$ from $\mathcal{A}^{\mathfrak{P}}_{\Des}$ to $\QSym$ yields an isomorphism $[\alpha]^{W}_{\Des} \rightarrow F_{\Comp(\Des(\alpha))}$ from $\mathcal{A}^{W}_{\Des}$ to $\QSym$.
\end{proof}

Like permutation statistics, many statistics on words can be related via symmetries, such as reversal, complementation, and
 reverse-complementation. Let $\alpha=(\alpha_1,\alpha_2,\ldots,\alpha_n)\in W_{n}$. The {\em reversal} of $\alpha$ is defined by $\alpha^r=(\alpha_n,\alpha_{n-1},\ldots,\alpha_1)$. The {\em complement} of $\alpha$, denoted by $\alpha^c$, is obtained by replacing the $i^{\text{th}}$ smallest letter in $\alpha$ by the $i^{\text{th}}$ largest letter in $\alpha$ for all valid $1\leq i\leq n$.  The {\em reverse-complement} of $\alpha$ is $\alpha^{rc}=(\alpha^r)^c=(\alpha^c)^r$. For example, if $\alpha=(5,1,4,4,2)$, then $\alpha^r=(2,4,4,1,5)$, $\alpha^c=(1,5,2,2,4)$ and $\alpha^{rc}=(4,2,2,5,1)$.

 In general, supposing $f$ is an involution on words that preserves the lengths, we write $\alpha^f$ to mean $f(\alpha)$ for $\alpha\in W_{n}$. For a given set of words $S$, let
 \[
 S^f=\{\{ \alpha^f:\alpha\in S\}\}.
 \] 
 
 Following Gessel and Zhuang \cite{Ges18}, we say $f$ is {\em shuffle-compatibility-preserving on words} if, 
 \begin{enumerate}
     \item for every pair of words $\gamma$ and $\kappa$, if $\pack(\gamma)=\pack(\kappa)$, then $\pack(\gamma^f)=\pack(\kappa^f)$; and
     \item for every pair of disjoint words $\alpha$ and $\beta$, there exists a pair of disjoint words $\bar{\alpha}$ and $\bar{\beta}$ such that $\pack(\alpha)=\pack(\bar{\alpha})$, $\pack(\beta)=\pack(\bar{\beta})$, and
 \[
(\alpha\shuffle\beta)^f=\bar{\alpha}^f\shuffle \bar{\beta}^f\quad\text{and}\quad (\bar{\alpha}\shuffle \bar{\beta})^f=\alpha^f\shuffle\beta^f.
 \]
 \end{enumerate}

We say two statistics on words $\St_1$, $\St_2$ are $f$-{\em equivalent} if $\St_1\circ f$ is equivalent to $\St_2$, that is, $\St_1(\alpha^f)=\St_1(\beta^f)$ if and only if $\St_2(\alpha)=\St_2(\beta)$. 
\begin{theorem}\label{thm:wilf}
  
    Let $f$ be shuffle-compatibility-preserving on words, and suppose $\St_1$ and $\St_2$ are $f$-equivalent statistics on words. If $\St_1$ is shuffle-compatible on words, then $\St_2$ is shuffle-compatible on words.  
    \end{theorem}

\begin{proof}
  Suppose $\{\alpha, \beta\}$ and $\{\kappa, \upsilon\}$ are two pairs of disjoint words in $W$ such that $\St_2(\alpha)=\St_2(\kappa)$, $\St_2(\beta)=\St_2(\upsilon)$, $|\alpha|=|\kappa|$, and $|\beta|=|\upsilon|$. By the assumption that $f$ is shuffle-compatibility-preserving, there exist $\bar{\alpha}$, $\bar{\beta}$, $\bar{\kappa}$ and $\bar{\upsilon}$ such that $\pack(\alpha)=\pack(\bar{\alpha})$, $\pack(\beta)=\pack(\bar{\beta})$, $\pack(\kappa)=\pack(\bar{\kappa})$,  $\pack(\upsilon)=\pack(\bar{\upsilon})$,  and
   \[
(\alpha\shuffle\beta)^f=\bar{\alpha}^f\shuffle \bar{\beta}^f,\quad (\bar{\alpha}\shuffle \bar{\beta})^f=\alpha^f\shuffle\beta^f,\quad\text{and}\quad (\kappa\shuffle\upsilon)^f=\bar{\kappa}^f\shuffle \bar{\upsilon}^f,\quad (\bar{\kappa}\shuffle \bar{\upsilon})^f=\kappa^f\shuffle\upsilon^f.
 \]

 By definition of shuffle-compatibility-preserving,  $\pack(\alpha^f)=\pack(\bar{\alpha}^f)$, which implies $\St_1(\alpha^f)=\St_1(\bar{\alpha}^f)$ by the shuffle-compatibility of $\St_1$. Since $\St_1$ and $\St_2$ are $f$-equivalent, we get  $\St_2(\alpha)=\St_2(\bar{\alpha})$.
Together with  $\St_2(\alpha)=\St_2(\kappa)$, we have
$\St_2(\bar{\alpha})=\St_2(\bar{\kappa})$, and so $\St_1(\bar{\alpha}^f)=\St_1(\bar{\kappa}^f)$. Similarly, one also has $\St_1(\bar{\beta}^f)=\St_1(\bar{\upsilon}^f)$.

 It follows from the shuffle-compatibility of $\St_1$ that
 \[
 \{\{ \St_1\gamma : \gamma\in\bar{\alpha}^f\shuffle\bar{\beta}^f \}\}=\{\{ \St_1\gamma:\gamma\in\bar{\kappa}^f\shuffle\bar{\upsilon}^f \}\}.
 \]
 By the fact that $\St_1$ and $\St_2$ are $f$-equivalent, this is equivalent to 
  \[
 \{\{ \St_2\gamma^f:\gamma\in\bar{\alpha}^f\shuffle\bar{\beta}^f \}\}=\{\{  \St_2\gamma^f:\gamma\in\bar{\kappa}^f\shuffle\bar{\upsilon}^f \}\}.
 \]
One can reform it and get
  \[
 \{\{ \St_2\gamma:\gamma^f\in\bar{\alpha}^f\shuffle\bar{\beta}^f \}\}=\{\{ \St_2\gamma:\gamma^f\in\bar{\kappa}^f\shuffle\bar{\upsilon}^f \}\}.
 \]
 Notice from $(\alpha\shuffle\beta)^f=\bar{\alpha}^f\shuffle \bar{\beta}^f$ and $(\kappa\shuffle\upsilon)^f=\bar{\kappa}^f\shuffle \bar{\upsilon}^f$, the equation above is equivalent to
  \[
 \{\{ \St_2\gamma:\gamma\in \alpha\shuffle\beta \}\}=\{\{ \St_2\gamma:\gamma\in \kappa\shuffle\upsilon\}\}.
 \]
 This shows that $\St_2$ is shuffle-compatible.
\end{proof}

We apply this concept to the reversal, complementation, and reverse-complementation maps.

 \begin{lemma}\label{lem:symmetry}
The reversal, complementation, and reverse-complementation maps
 are shuffle-compatibility-preserving on words.    
 \end{lemma}
 \begin{proof}
     For all three maps, it is straightforward to check that $\pack(\pi^f)=\pack(\pi)^f$. Hence if $\pack(\gamma)=\pack(\kappa)$  then $\pack(\gamma^f)=\pack(\kappa^f)$ for any words $\gamma, \kappa$.
     For reversal, it is clear that $(\alpha\shuffle\beta)^r=\alpha^r\shuffle\beta^r$ for every disjoint pair of words $\{\alpha, \beta\}$. It follows that $(\alpha^r\shuffle\beta^r)^r=(\alpha^r)^r\shuffle(\beta^r)^r=\alpha\shuffle\beta$. This proves that reversal is shuffle-compatibility-preserving. 

As for complementation, in general $(\alpha\shuffle\beta)^c\not=\alpha^c\shuffle\beta^c$. For a disjoint pair of words $\{\alpha, \beta\}$, let $K=\{\alpha_1,\ldots,\alpha_{|\alpha|},\beta_1,\ldots,\beta_{|\beta|}\}$ be the set of letters that appear in $\alpha$ and $\beta$. Let $\omega:K\to K$ be the map that sends the smallest $i^\text{th}$  letter to the largest $i^\text{th}$ letter in $K$ for each $i$. Let $\omega(\alpha)=\omega(\alpha_1)\omega(\alpha_2)\cdots\omega(\alpha_{|\alpha|})$ by applying $\omega$ to each letter of $\alpha$. Similarly, define $\omega(\beta)=\omega(\beta_1)\omega(\beta_2)\cdots\omega(\beta_{|\beta|})$. Note that $\pack(\alpha)=\pack(\omega(\alpha^c))$, $\pack(\beta)=\pack(\omega(\beta^c))$. To prove that complementation is shuffle-compatibility-preserving, it suffices to show that
     \[     (\alpha\shuffle\beta)^c=\omega(\alpha^c)^c\shuffle \omega(\beta^c)^c\quad\text{and}\quad (\omega(\alpha^c)\shuffle \omega(\beta^c))^c=\alpha^c\shuffle\beta^c.
     \]
     Since $\omega$ is a bijection, we only need to show $(\alpha\shuffle\beta)^c=\omega(\alpha^c)^c\shuffle \omega(\beta^c)^c$. Equivalently, it suffices to show $(\alpha\shuffle\beta)^c=\omega(\alpha)\shuffle \omega(\beta)$ by the observation that $\omega(\alpha^c)^c=\omega(\alpha)$, $\omega(\beta^c)^c=\omega(\beta)$.

     For one direction, given $\kappa\in\alpha\shuffle\beta$, to show that $\kappa^c\in \omega(\alpha)\shuffle \omega(\beta)$, it suffices to show that $\kappa^c$ contains both $\omega(\alpha)$ and $\omega(\beta)$ as subsequences. In fact, when taking complement, the subsequence $\alpha$ and $\beta$ in $\kappa$ becomes $\omega(\alpha)$ and $\omega(\beta)$ respectively. The other direction can be proved using the same idea. 
 \end{proof}

 In the case of permutations, the ascent set is uniquely determined by the descents, but given the presence of ties in these more general words, this is no longer the case. However, $\Des$ and $\Asc$ are $c$-equivalent, so by Theorem~\ref{thm:SCofDes} and Lemma~\ref{lem:symmetry} we have the shuffle-compatibility of $\Asc$ on words. Additionally, Proposition~\ref{prop:perm_subalg} implies that $\mathcal{A}_{\Asc}^{W}$ is isomorphic to $\mathcal{A}^{\mathfrak{P}}_{\Asc}$ and thus also to $\mathcal{A}^{\mathfrak{P}}_{\Des}$, which is isomorphic to $\QSym$, as shown in~\cite{Ges18}.

\begin{corollary}\label{cor:Asc} (Shuffle-compatibility of the ascent set)
\begin{enumerate}
    \item The ascent set statistic $\Asc$ is shuffle-compatible on words.
    \item The linear map \[[\alpha]_{\Asc} \mapsto F_{\Comp(\Asc(\alpha^c))},\] is an algebra isomorphism between $\mathcal{A}^W_{\Asc}$ and $\QSym$.
\end{enumerate}
\end{corollary}

Note that the identity map $\id$ is certainly shuffle-compatibility-preserving on words. So, if one statistic is uniquely defined by the values of another, we might say they are $\id$-equivalent on words. For example, the set of weak ascents (equivalent to the union of the set of ascents and the set of ties) of a word is completely determined by the set of descents for that word. So, by Theorem~\ref{thm:wilf}, the set of weak ascents and the set of weak descents are both shuffle-compatible on words because they are
$\id$-equivalent on words to $\Des$ and $\Asc$.
In the following theorem, we make more general statements about statistics that can be determined from the descent set specifically.

\begin{theorem}\label{thm:ptow}
    Suppose $\St$ is a statistic on words such that
    \begin{enumerate}
        \item $\St$ depends on $\Des$, that is, if two words $\alpha$ and $\beta$ with the same length satisfy $\Des(\alpha)=\Des(\beta)$, then $\St(\alpha)=\St(\beta)$, and
        \item $\St$ is shuffle-compatible on permutations $\mathfrak{P}$. 
    \end{enumerate}
    Then, $\St$ is shuffle-compatible on words.
\end{theorem}
\begin{proof}
A word $\alpha$ can always be associated with a permutation $\Tilde{\alpha}$ (not unique in general) of the same length such that $\Des(\alpha)=\Des(\Tilde{\alpha})$. One can similarly construct a permutation $\Tilde{\beta}$ for the word $\beta$ satisfying $\Des(\beta)=\Des(\Tilde{\beta})$. By the shuffle-compatibility on words of $\Des$ from Theorem~\ref{thm:SCofDes},  $\Des(\alpha\shuffle\beta)=\Des(\Tilde{\alpha}\shuffle\Tilde{\beta})$.
It follows from condition (1) that $\St(\alpha)=\St(\Tilde{\alpha})$, $\St(\beta)=\St(\Tilde{\beta})$ and $\St(\alpha\shuffle\beta)=\St(\Tilde{\alpha}\shuffle\Tilde{\beta})$.

Suppose $\{\alpha, \beta\}$ and $\{\kappa, \upsilon\}$ are two pairs of disjoint words in $W$ such that $\St_2(\alpha)=\St_2(\kappa)$, $\St_2(\beta)=\St_2(\upsilon)$, and $|\alpha|=|\kappa|$, $|\beta|=|\upsilon|$. Take corresponding permutations $\Tilde{\alpha}$, $\Tilde{\beta}$, $\Tilde{\kappa}$, and $\Tilde{\upsilon}$. To prove that $\St(\alpha \shuffle \beta)=\St(\kappa\shuffle\upsilon)$, it suffices to show $\St(\Tilde{\alpha}\shuffle\Tilde{\beta})=\St(\Tilde{\kappa}\shuffle\Tilde{\upsilon})$. This is guaranteed by the fact that $\Tilde{\alpha}$, $\Tilde{\beta}$, $\Tilde{\kappa}$, and $\Tilde{\upsilon}$ are permutations, as well as the condition (2) that $\St$ is shuffle-compatible on permutations. 
\end{proof}

Theorem \ref{thm:ptow} gives us the shuffle-compatibility for an analogue to the peak statistic, which is itself not shuffle-compatible on words. Recall that $\QSym$ has a subalgebra $\mathbf{\Pi}$ called Stembridge's algebra of peaks \cite{stembridge1997enriched}. This subalgebra is spanned by the peak quasisymmetric functions $K_{n, S}$ where $S$ is the peak set of some permutation. The product of peak quasisymmetric functions is given by $K_{m, \Peak(\tau)} K_{n, \Peak(\theta)} = \sum_{\pi \in \tau \shuffle \theta} K_{m+n, \Peak(\pi)}$ where $\tau$ and $\theta$ are disjoint permutations \cite{Ges18}.

\begin{corollary}\label{cor:cliff}(Shuffle-compatibility of the cliff set)
    \begin{enumerate}
        \item The cliff set statistic $\Cliff$ is shuffle-compatible on words.
        \item The linear map on $\mathcal{A}^{W}_{\Cliff}$ defined by \[[\alpha]_{\Cliff} \mapsto K_{|\alpha|, \Cliff(\alpha)},\] is an algebra isomorphism from $\mathcal{A}^{W}_{\Cliff}$ to $\mathbf{\Pi}$. 
    \end{enumerate}
\end{corollary}

\begin{proof}
    First, observe that $\Cliff$ is determined by $\Des$ on words, and that $\Cliff$ restricts to $\Peak$ on permutations, which is shuffle-compatible \cite{Ges18}. Thus, Theorem \ref{thm:ptow} applies.

    Next, the cliff set of any word is the peak set of some permutation, and so by Proposition \ref{prop:perm_subalg}, we have that $\mathcal{A}_{\Cliff}^{W}$ is isomorphic to $\mathcal{A}_{\Peak}^{\mathfrak{P}}$. Specifically, this isomorphism takes the form of the map $[\alpha]_{\Cliff}^{W} \mapsto [\pi]_{\Peak}^{\mathfrak{P}}$ where $[\pi]_{\Peak}^{\mathfrak{P}} \subseteq  [\alpha]_{\Cliff}^{W}$. Composing this with the isomorphism $[\pi]_{\Peak}^{\mathfrak{P}} \mapsto K_{|\pi|, \Peak(\pi)}$ from \cite[Theorem 4.7]{Ges18} yields the desired isomorphism from $\mathcal{A}^{W}_{\Cliff}$ to $\mathbf{\Pi}$.
    \end{proof}

The shuffle-compatibility of various permutation statistics extended to words follows from Theorem~\ref{thm:ptow}. The algebraic results mirror those in \cite{Ges18}, so we do not restate them here. 

\begin{corollary}\label{cor:des,asc,maj}
The descent number statistic $\des$, ascent number statistic $\asc$, major index statistic $\maj$, and cliff number statistic $\cliff$ are shuffle-compatible on words. 
\end{corollary}

The tie set statistic is independent of the descent set statistics and thus requires a different technique.  This statistic is of special interest as it does not have an analogue for permutations.

\begin{theorem}\label{thm:Tie}(Shuffle-compatibility of the tie set)
\begin{enumerate}
    \item The tie set statistic $\Tie$ is shuffle-compatible on words.
    \item The tie set shuffle algebra $\mathcal{A}^{W}_{\Tie}$ is isomorphic to $
    \QSym$.
    \end{enumerate}
\end{theorem}

\begin{proof}
    Consider two disjoint words $\alpha$ and $\beta$ with $|\alpha|=m$, $|\beta|=n$,  $\Tie(\alpha) = A$ and $\Tie(\beta) = B$. Each $\kappa \in \alpha \shuffle \beta$ can be associated with two sets $X \sqcup Y = [m+n]$, and from this, $\Tie(\kappa)$ can be determined. Let $ X = \{x_1 < x_2 < \cdots < x_m \}$ and $Y = \{y_1 < y_2 < \cdots < y_n\}$ where $x_i$ is the location of $\alpha_i$ in $\kappa$ and $y_i$ is the location of $\beta_i$ in $\kappa$.

    Now, consider each entry $\kappa_i$ for $1 \leq i \leq m+n-1$. If $i = x_j \in X$, then $i \in \Tie(\kappa)$ if and only if $i+1 = x_{j+1} \in X$ and $j \in \Tie(\alpha)$. In this case, a tie in $\alpha$ is preserved during the shuffle. If $i = y_j \in Y$, then $i \in \Tie(\kappa)$ if and only if $i+1 = y_{j+1} \in Y$ and $j \in \Tie(\beta)$. This is the case where a tie is preserved from $\beta$.
    There will never be a tie between an element from $\alpha$ and an element from $\beta$ since they share no common entries. Since we can determine the tie set for each shuffle of $\alpha$ and $\beta$ in this way, we can determine $\{\{ \Tie(\kappa) : \kappa \in \alpha \shuffle \beta\}\}$ using only the same information, which is $|\alpha|$, $|\beta|$, $\Tie(\alpha)$, and $\Tie(\beta)$. 

    To show that $\mathcal{A}_{\Tie}^{W}$ is isomorphic to QSym, we first show that it admits a Hopf algebra structure. We give $\mathcal{A}_{\Tie}^{W}$ the coproduct $\Delta: \mathcal{A}_{\Tie}^{W} \rightarrow \mathcal{A}_{\Tie}^{W} \times \mathcal{A}_{\Tie}^{W}$ defined by $\Delta([\alpha]_{\Tie}^W) = \sum_{i=0}^n [(\alpha_1, \ldots, \alpha_i)]_{\Tie}^W \otimes [(\alpha_{i+1}, \ldots, \alpha_n)]^W_{\Tie}$. It is straightforward to check that this coproduct is well defined regardless of the choice of $\alpha$. {We define the unit $u: \mathbb{Q}\rightarrow \mathcal{A}_{\Tie}^{W}$ and counit $\varepsilon: \mathcal{A}_{\Tie}^W \rightarrow \mathbb{Q}$ to be \[ u(k) = k \cdot \textbf{1} \text{\quad and \quad}\varepsilon([\alpha]_{\Tie}^W)=\begin{cases}
        1 & \text{ if }[\alpha]_{\Tie}^W=\mathbf{1}\\ 
        0 & \text{ otherwise}
    \end{cases},\] respectively, where $\mathbf{1}$ is the multiplicative identity element associated with the empty word.  First, we check the compatibility of the unit with the coproduct. Observe that, for any $[\alpha]_{\Tie}^W, [\beta]^W_{\Tie}$,
    \[\mu \circ (\varepsilon \otimes \varepsilon)([\alpha]_{\Tie}^W \otimes [\beta]_{\Tie}^W) = \varepsilon \circ \mu([\alpha]_{\Tie}^W \otimes [\beta]_{\Tie}^W) =  \begin{cases}
        1 & \text{ if }[\alpha]_{\Tie}^W=[\beta]_{\Tie}^W = \textbf{1}\\
        0 & \text{ otherwise}
    \end{cases},\] as desired. Second, we check the compatibility of the coproduct with the unit. Observe that, for $k \in \mathbb{Q}$, \[\Delta \circ u(k) = (u \otimes u)\circ \Delta(k) = k (\textbf{1} \otimes \textbf{1}),\] as desired. Third, we check the compatibility of the unit and counit. Observe that $\varepsilon \circ u(k) = \varepsilon(k \cdot \textbf{1}) = k = \id(k)$, as desired.}
    
    It remains to verify that the product (which we denote here with $\mu: \mathcal{A}_{\Tie}^{W} \otimes \mathcal{A}_{\Tie}^{W} \rightarrow \mathcal{A}_{\Tie}^{W}$) and coproduct meet the conditions for a Hopf algebra. Given a word $\alpha$, let $\alpha^{(i)}=(\alpha_1, \ldots, \alpha_i)$ and $\bar{\alpha}^{(i)}=(\alpha_{i+1}, \ldots, \alpha_n)$.
    For $\alpha\in W_n, \beta \in W_m$, observe that
    \[\Delta \circ \mu([\alpha]_{\Tie}^W \otimes [\beta]_{\Tie}^W) = \Delta(\sum_{\gamma \in\alpha  \shuffle \beta} [\gamma]^W_{\Tie})= \sum_{\gamma \in \alpha \shuffle \beta,}\sum_{k=0}^{n+m} [\gamma^{(k)}]^W_{\Tie} \otimes [\bar{\gamma}^{(k)}]_{\Tie}^W.\]
    On the other hand, 
    \begin{align*}
        (\mu \otimes \mu) & \circ (\id \otimes T \otimes \id ) \circ(\Delta \circ \Delta)([\alpha]^W_{\Tie} \otimes [\beta]^W_{\Tie})\\ 
        &= (\mu \otimes \mu) \circ (\id \otimes T \otimes \id)((\sum_{i=0}^{n} [ \alpha^{(i)}]^W_{\Tie} \otimes [\bar{\alpha}^{(i)}]^W_{\Tie})\otimes (\sum_{j=0}^{m} [\beta^{(j)}]^W_{\Tie} \otimes [\bar{\beta}^{(j)}]^W_{\Tie}))\\
        &= (\mu \otimes \mu) \circ (\id \otimes T \otimes \id)\sum_{i=0,}^{n}\sum_{j=0}^{m} [ \alpha^{(i)}]^W_{\Tie} \otimes [\bar{\alpha}^{(i)}]^W_{\Tie} \otimes [\beta^{(j)}]^W_{\Tie} \otimes [\bar{\beta}^{(j)}]^W_{\Tie}\\
        &= (\mu \otimes \mu) \sum_{i=0,}^{n}\sum_{j=0}^{m} [ \alpha^{(i)}]^W_{\Tie} \otimes [\beta^{(j)}]^W_{\Tie} \otimes [\bar{\alpha}^{(i)}]^W_{\Tie}  \otimes [\bar{\beta}^{(j)}]^W_{\Tie}\\
        &= \sum_{i=0,}^{n}\sum_{j=0,}^{m} \sum_{\gamma' \in \alpha^{(i)}\shuffle \beta^{(j)},} \sum_{\gamma'' \in \bar{\alpha}^{(i)} \shuffle \bar{\beta}^{(j)}} [\gamma']_{\Tie}^W \otimes [\gamma'']_{\Tie}^W.
    \end{align*}
    Each of the words $\gamma^{(k)}$ in the first expression is some shuffle of some first numbers of letters of $\alpha$ and some first numbers of letters of $\beta$, while $\bar{\gamma}^{(k)}$ makes up a shuffle of the remaining letters. Thus, these match up exactly with the words $\gamma'$ and $\gamma''$ in the second expression. Given that the two expressions are equivalent, $\mathcal{A}_{\Tie}^W$ meets the conditions for a bialgebra \cite[Definition 1.3.7]{grinberg}. Since it is also graded and connected, $\mathcal{A}_{\Tie}^W$ is a Hopf algebra \cite[Lemma 2.1]{ehrenborg1996posets}.

    Now, since $\mathcal{A}^{W}_{\Tie}$ is a commutative Hopf algebra, its dual is a cocommutative Hopf algebra. Given that the dimension of the $n^{\text{th}}$ graded component of the dual is the number of integer compositions of $n$, the dual is isomorphic to $\NSym$ by \cite[Theorem 12]{aliniaeifard2021hopfstructuresrepresentationtheory}. 
    Since the dual is isomorphic to $\NSym$, the algebra $\mathcal{A}_{\Tie}^W$ must be isomorphic to $\QSym$.
\end{proof}

Given the isomorphism between $\mathcal{A}^{W}_{\Tie}$ and $\QSym$, one can construct a basis of $\QSym$ analogous to the basis of $\mathcal{A}^{W}_{\Tie}$ made up by the $\Tie$-equivalence classes, which we call the \emph{tie basis}. For a word $w$,  define the tie composition of $w$ to be $ \Comp(\Tie(w))$. For a tie composition $\alpha \vDash n$, let the tie word of $\alpha$, denoted $w^{\Tie}_{\alpha} = (w_1, w_2, \ldots, w_n)$, be the unique word such that $w_1 = 1$, and $w_{i} = w_{i+1}$ if $i \in \Set(\alpha)$ but $w_{i} + 1 = w_{i+1}$ if $i \not\in \Set(\alpha)$ for $2 \leq i \leq n-1$. Note that we have constructed the tie word of $\alpha$ so that $\Comp(\Tie(w^{\Tie}_{\alpha})) = \alpha$.

\begin{theorem}\label{thm:T_basis}
    There exists a basis $\{T_{\alpha}\}_{\alpha}$ of $\QSym$, indexed by integer compositions $\alpha$, with the product\[T_{\alpha}T_{\beta} = \sum_{v \in w^{\Tie}_{\alpha} \shifted w^{\Tie}_{\beta}} T_{\Comp(\Tie(v))}=\sum_{\gamma \in \vDash m+n} c^{\gamma}_{\alpha, \beta}T_{\gamma},\] for $\alpha \in W_m, \beta \in W_n$, where $c^{\gamma}_{\alpha,\beta}$ is the number of words with Tie composition $\gamma$ in the set of shifted shuffles of $w^{\Tie}_{\alpha}$ and $w^{\Tie}_{\beta}$.
\end{theorem}

\begin{example}
We will calculate $T_{(1,2)} T_{(1)}$. The compositions $(1,2)$ and $(1)$ correspond to the tie sets $\{1\}$ and $\emptyset$ and the tie words $w^{\Tie}_{(1,2)}=(1,1,2)$ and $w^{\Tie}_{(1)} = (1)$.  The shifted shuffles of these words are $\{\{(1,1,2,3),(1,1,3,2), (1,3,1,2), (3,1,1,2)\}\}$ which have tie compositions $\{\{(1,3), (1,3), (4),  (2,2)\}\}$. Thus, we have $T_{(1,2)}  T_{(1)}= 2T_{(1,3)} + T_{(2,2)} + T_{(4)}$.
\end{example}

\begin{lemma}
The number of shuffles of $\alpha \in W_n$ and $\beta \in W_m$ with $i$ ties between letters from $\alpha$ and $j$ ties  between letters from $\beta$ such that the letters from $\beta$ are in $k$ contiguous pieces is given by \[\binom{\tie(\beta)}{j}\binom{m-\tie(\beta)-1}{k+j-\tie(\beta)-1}\binom{\tie(\alpha)}{i}\binom{n+1-\tie(\alpha)}{k+i-\tie(\alpha)}.\]
\end{lemma}
\begin{proof}
    Let $\alpha \in W_n$ and $\beta \in W_m$.  We will count all shuffles of $\alpha$ and $\beta$ such that there are $i$ ties between letters, there are $j$ ties between letters from $\beta$, and the letters from $\beta$ are in $k$ contiguous pieces within the shuffle. To break $\beta$ into $k$ pieces, we must make $k-1$ breaks. The number of ways to break $\beta$ in $k$ pieces with $j$ ties among them is counted by first choosing $j$ ties to preserve and breaking $\beta$ along the remaining $\tie(\beta)-j$ ties. Then, choose $(k-1)-(\tie(\beta)-j)$ of the $m-1-\tie(\beta)$ spots between two letters that do not tie at which to break $\beta$. This yields \[\binom{\tie(\beta)}{j}\binom{m-\tie(\beta)-1}{(k-1)-(\tie(\beta)-j)}\] ways to break $\beta$ into $k$ pieces with a total of $j$ ties between them. 
    Next, we choose $i$ ties in $\alpha$ to preserve so that the remaining $\tie(\alpha)-i$ ties in $\alpha$ will be split by a part of $\beta$ in the shuffle. We then select $k-(\tie(\alpha)-i)$ spots for the remaining pieces of $\beta$ from the $n+1-\tie(\alpha)$ possible spots a piece of $\beta$ could be shuffled into. There are \[\binom{\tie(\alpha)}{i}\binom{n+1-\tie(\alpha)}{k-(\tie(\alpha)-i)}\] ways to make this choice. With these $k$ spots selected, we simply insert the $k$ pieces of $\beta$ in order to obtain a shuffle of $\alpha$ and $\beta$ as desired. 
\end{proof}

Given the lemma above, the total number of ties across all shuffles of any two words $\alpha$ and $\beta$ can also be determined only from $\tie(\alpha)$, $\tie(\beta)$, $|\alpha|$, and $|\beta|$.
\begin{theorem}\label{thm:tie}
    The tie number statistic $\tie$ is shuffle-compatible on words.
\end{theorem}

We leave it as an open problem to find an expansion of the tie basis in terms of other bases of $\QSym$ or to find a polynomial expression for tie quasisymmetric functions. A polynomial expansion, specifically, may also yield an algebraic interpretation for the shuffle algebra of the tie number.

\section{Parking functions}\label{sec:parking}

Parking functions comprise a subset of words with close links to Hopf algebras and various well-studied statistics. We provide a brief review of parking functions following \cite{Yan15}.

\begin{definition}
    A \emph{parking function} of length $n$ is a sequence $\alpha = (\alpha_1, \alpha_2, \ldots, \alpha_n)$  such that the nondecreasing rearrangement $\alpha_1'  \leq \alpha_2' \leq \cdots \leq \alpha'_n$ has the property that $\alpha_i' \leq i$. We denote the set of parking functions of length $n$ by $\PF_{n}$ and let $\PF = \cup_{n\geq 0} \PF_{n} $.
\end{definition}

Parking functions also have the following interpretation. Imagine a one-way street with $n$ spots and $n$ cars lined up to park in those spots. Each car has a preferred spot, and the sequence of these preferences is $\alpha = (\alpha_1, \ldots, \alpha_n)$ where car $i$ prefers spot $\alpha_i$.  The cars will park one by one, driving to their preferred spot and then taking the first available spot. If there are no unoccupied spots remaining past their preference, the car will be forced to leave. The preference sequence $\alpha$ is called a parking function if all cars can park.

In Table \ref{tab:pfstats} below, we list statistics that we show to be weakly shuffle-compatible on parking functions. Let $\alpha = (\alpha_1, \alpha_2, \ldots, \alpha_n) \in \PF_n$ and let $s_i$ be the spot car $i$ parks in after the parking process on $\alpha$.

\begin{table}[h!]
    \centering
    \begin{tabular}{|c|c|c|} \hline
    \textbf{Name} & \textbf{Definition} & \textbf{FindStat}\\ \hline
     Inversion set & $\Inv(\alpha) = \{(i,j): i < j, \alpha_i > \alpha_j\}$ &  \\ \hline
        Inversion number & $\inv(\alpha) = |\Inv(\alpha)|$  & St000018 \\  \hline
        Outcome  & $\out(\alpha) = (s_1, s_2, \ldots, s_n)$ & \\ \hline
         Displacement sequence & $\Disp(\alpha) =(s_1 - \alpha_1, s_2 - \alpha_2, \ldots, s_n - \alpha_n)$ & \\ \hline
        Displacement   & $\disp(\alpha) = \sum_{i=1}^n s_i - \alpha_i = \binom{n+1}{2} - \sumst(\alpha)$  & St000188 \\ \hline
         Maximum displacement of single car & $\mathsf{maxDisp}(\alpha) = \max\{ s_1-\alpha_1, s_2 - \alpha_2, \ldots, s_n -\alpha_n\}$ & St000943 \\ \hline
        Lucky car set & $\Lucky(\alpha) = \{i : s_i = \alpha_i\}$  &  \\ \hline
        Lucky car number   &  $\lucky(\alpha) = |\Lucky(\alpha)|$ & St000135 \\ \hline
        Unlucky car set & $\Unlucky(\alpha) =\{i : s_i \not= \alpha_i\}$  & \\ \hline
        Unlucky car number  & $\unlucky(\alpha) = |\Unlucky(\alpha)|$ &    \\ \hline
       Weakly increasing rearrangement & $\winc(\alpha) = (\alpha_1', \alpha_2', \ldots, \alpha_n')$  &  \\ \hline
        Leading elements 
        & $\lel(\alpha) = |\{i : \alpha_i = \alpha_1, 1 \leq i\}|$ & \\ \hline
        Number of ones  & $\ones(\alpha) = |\{i : \alpha_i = 1\}|$  & \\ \hline
        Size of initial strictly increasing segment  & $\isi(\alpha) = \max\{i: \alpha_1 < \alpha_2 < \cdots < \alpha_i\}$  & St001804\\ \hline
        Sum of entries & $\sumst(\alpha) = \alpha_1 + \alpha_2 + \cdots + \alpha_n$ & St000165 \\ \hline
        Sum of entries minus length & $\mathsf{sml}(\alpha) = \sumst(\alpha) - |\alpha|$ & St000540 \\ \hline
    \end{tabular}
    \caption{Statistics on parking functions \cite{COLARIC2021102129, durmic2023probabilistic, gessel2004refinement, FindStat, stanleyPF,  Yan15}}
    \label{tab:pfstats}
\end{table}

The Hopf algebraic structure on parking functions was introduced by Novelli and Thibon in~\cite{Nov04}. The Hopf algebra $\PQSym$ employs the same shifted shuffle as $\FQSym$ with a slight difference in coproduct.
Note that in \cite{Nov04}, the shuffle product we denote with $\shifted$ is instead denoted $\Cup$. The product on the fundamental basis $\{F_{\alpha}\}_{\alpha \in \PF_n}$ of $\PQSym$ is given by \[ F_{\alpha} F_{\beta} = \sum_{\gamma \in \alpha \shifted \beta} F_{\gamma}.\]

$\PQSym$ is self-dual and contains $\FQSym$ as a subalgebra. It also contains the \emph{Catalan quasisymmetric functions} \cite{Nov04}. This algebra, $\CQSym$ has basis $\{\mathbf{P}_{\alpha}\}_{\alpha \in \WIPF}$ where $\WIPF$ is the set of weakly increasing parking functions. The product on this basis is given by $\mathbf{P}_{\alpha}\mathbf{P}_{\beta} = \mathbf{P}_{\alpha \cdot \beta[n]},$ for $\alpha \in \WIPF_n, \beta \in \WIPF_m$, where $\alpha \cdot \beta[n]$ is the concatenation of $\alpha$ and $\beta[n]$.

\subsection{Weak shuffle-compatibility for parking functions and quotients of $\PQSym$}

The following generalization of weak shuffle-compatibility to parking functions follows naturally by expanding Definition~\ref{def:weakperm} to encompass parking functions.

\begin{definition}
A statistic $\St$ is \emph{weakly shuffle-compatible on parking functions} if $\{\{ \St(\gamma) : \gamma \in \alpha \shifted \beta\}\}$ is determined by $|\alpha|$, $|\beta|$, $\St(\alpha)$, and $\St(\beta)$ for any two parking functions $\alpha$ and $\beta$.
\end{definition}
Equivalently, $\St$ is weakly shuffle-compatible on parking functions if \[\{\{\St(\gamma) : \gamma \in \alpha \shifted \beta\}\}=\{\{\St(\gamma) : \gamma \in \kappa \shifted \upsilon\}\},\] for any two pairs of parking functions $\{\alpha, \beta\}$ and $\{\kappa, \upsilon\}$ where  $|\alpha|=|\kappa|$, $|\beta|=|\upsilon|$, $\St(\alpha) = \St(\kappa)$ and $\St(\beta)=\St(\upsilon)$. 

Let $[\alpha]^{\PF}_{\St} = \{\beta \in \PF : \St(\alpha)=\St(\beta), |\alpha|=|\beta|\}$ be the $\St$ equivalence class of $\alpha \in \PF$, where we may drop the $\PF$ sueprscript when it is clear from context. Let $\PF/\St$ be the set of all such equivalence classes.

\begin{definition}
The $\St$ shifted shuffle algebra, denoted $\mathcal{A}^{\PF}_{\St}$, is defined as the algebra with a basis formed by $\St$-equivalence classes of parking functions with the product \[ [\alpha]^{\PF}_{\St} [\beta]^{\PF}_{\St} = \sum_{ \gamma \in \alpha \shifted \beta} [\gamma]^{\PF}_{\St}.\]
\end{definition}

\begin{proposition}\label{prop:alg_pqsym}
    If $\St$ is weakly shuffle-compatible on parking functions, then $\mathcal{A}_{\St}^{\PF}$ is a quotient of $\PQSym$.
\end{proposition}

\begin{proof}
 Let $\St$ be weakly shuffle-compatible on parking functions. Consider the map $\varsigma: \PQSym \rightarrow \mathcal{A}^{\PF}_{\St}$, defined by $F_{\alpha} \rightarrow [\alpha]_{\St}^{\PF}$. The map is clearly surjective. Additionally, \[\varsigma(F_{\alpha}F_{\beta}) = \varsigma(\sum_{\gamma \in \alpha \shifted \beta} F_{\gamma})= \sum_{\gamma \in \alpha \shifted \beta} [\gamma]^{\PF}_{\St} = [\alpha]^{\PF}_{\St}[\beta]^{\PF}_{\St}=\varsigma(F_{\alpha})\varsigma(F_{\beta}).\] Since $\varsigma$ is a surjective morphism, we have that $\mathcal{A}^{\PF}_{\St}$ is isomorphic to $\PQSym/\ker(\varsigma)$.
\end{proof}

Note that, given permutations are a subset of parking functions, if a statistic is weakly shuffle-compatible on parking functions, it is also weakly shuffle-compatible on permutations. 
It is straightforward to check that any statistic $\St$ that is shuffle-compatible on words is weakly shuffle-compatible on parking functions. In many of these cases, the algebraic results translate over exactly, and so we do not restate them here.

\begin{corollary}
    The statistics $\Des$, $\Asc, \Tie, \Cliff, \des, \asc, \maj$ and $\cliff$ are weakly shuffle-compatible on parking functions.
\end{corollary}

Our primary interest in this section is statistics that are not shuffle-compatible on words.

\begin{theorem}\label{thm:Inv}(Weak shuffle-compatibility of the inversion set)
\begin{enumerate}
    \item The set of inversions statistic $\Inv$ is weakly shuffle-compatible on parking functions. 
    \item The linear map on $\mathcal{A}^{\PF}_{\Inv}$ defined by \[[\alpha]^{\PF}_{\Inv} \rightarrow \mathbf{F}_{\alpha},\] where $\alpha$ is the unique permutation of $\mathfrak{S}_{|\alpha|}$ in $[\alpha]^{\PF}_{\Inv}$, is an algebra isomorphism from $\mathcal{A}^{\PF}_{\Inv}$ to $\FQSym$.
\end{enumerate}
      
\end{theorem}
  
\begin{proof}
    Let $\alpha$ and $\beta$ be two parking functions with inversion sets $\Inv(\alpha)$ and $\Inv(\beta)$. Consider a parking function $\gamma \in \alpha \shifted \beta$ where the indices of $\alpha$ in $\gamma$ are contained in the set $X = \{x_1 < x_2 < \cdots\}$ and the indices of $\beta$ in $\gamma$ are contained in the set $Y = \{y_1 < y_2 < \cdots\}$.  Then, a pair $(i,j)$ with $i<j$ is an inversion in $\gamma$ if and only if one of the following is true \begin{itemize}
        \item $i=x_g,j=x_h \in X$ and $(g,h) \in \Inv(\alpha)$,
        \item $i = y_g, j=y_h \in Y$ and $(g,h) \ \in \Inv(\beta)$, or 
        \item $i \in Y$ and $j \in X$.
    \end{itemize}
    Thus, the inversion sets of all shuffles of two words can be determined only from their lengths and inversion sets, and so $\Inv$ is shuffle-compatible on parking functions. 

 Next, we consider the algebraic part of the statement. The inversion set of any parking function $\alpha$ is the inversion set of some permutation in $\mathfrak{S}_{|\alpha|}$. Thus, each equivalence class can be associated with a unique permutation from $\mathfrak{S}_{|\alpha|}$  since permutations in $\mathfrak{S}_{|\alpha|}$ are uniquely defined by their inversion sets. Let $\varsigma: \mathcal{A}_{\Inv}^{\PF} \rightarrow \FQSym$ be defined by $\varsigma([\alpha]_{\Inv}) = \mathbf{F}_{\alpha}$ where $\alpha$ is picked to be the permutation from $\mathfrak{S}_{|\alpha|}$ that must be in   $\varsigma([\alpha]_{\Inv})$.  
Observe that \[\varsigma([\alpha]_{\Inv}[\beta]_{\Inv}) = \varsigma(\sum_{\gamma \in \alpha \shifted \beta} [\gamma]_{\Inv}) = \sum_{\gamma \in \alpha \shifted \beta} \mathbf{F}_{\gamma} = \mathbf{F}_{\alpha}\mathbf{F}_{\beta} = \varsigma([\alpha]_{\Inv})\varsigma([\beta]_{\Inv}).\]
From there, the result follows. 
\end{proof}

\begin{corollary}\label{cor:inv_WQSym} 
The number of inversions statistic $\inv$ is weakly shuffle-compatible on parking functions.
\end{corollary}

\begin{proof}
The compatibility of the inversion number can be determined from the process in the proof above. That is, every inversion in $\alpha$ and $\beta$ will be preserved in each shuffle, and then additional inversions will be added whenever elements from $\alpha$ are placed to the right of elements of $\beta$. 
\end{proof}

Parking functions have many interesting statistics that do not extend to words in general, and many of them prove to be weakly shuffle-compatible.

\begin{theorem}\label{thm:outcome} (Weak shuffle-compatibility of the outcome)
\begin{enumerate}
    \item The outcome statistic $\out$ is weakly shuffle-compatible on parking functions. 
    \item The linear map on $\mathcal{A}^{\PF}_{\out}$ defined by \[[\alpha]_{\out} \mapsto \mathbf{F}_{\out(\alpha)},\] is an algebra isomorphism between $\mathcal{A}^{\PF}_{\out}$ and $\FQSym$.
\end{enumerate}
    
\end{theorem}

\begin{proof} Consider some $\alpha \in \PF_n$, $\beta \in \PF_m$, and $\gamma \in \alpha \shifted \beta$ where indices $\{i_1 < \cdots <  i_n\}$ correspond to entries originally coming from $\alpha$ and indices $\{j_1 < \cdots < j_m\}$ correspond to entries originally coming from $\beta$. That is, $\gamma_{i_t} = \alpha_t$ and $\gamma_{j_t} = \beta_t$.  Further, let $\out(\alpha) = (s_1^{\alpha}, s_2^{\alpha}, \ldots, s_n^{\alpha})$ and $\out(\beta) = (s_1^{\beta}, s_2^{\beta}, \ldots, s_m^{\beta})$. In $\gamma$, the cars coming from $\alpha$ will park in the spots $1$ through $n$, specifically, car $i_t$ in $\gamma$ will park in the same spot as car $t$ in $\alpha$.  The cars coming from $\beta$ will park in the spots $n+1$ through $n+m$, specifically, if car $t$ in $\beta$ parks in spot $p$, then car $j_t$ in $\gamma$ will park in spot $n+p$. Thus, $\out(\gamma) = (s^{\gamma}_{1}, \ldots, s_{n+m}^{\gamma})$ will be given by $s^{\gamma}_{i_t} = s^{\alpha}_{t}$ and $s^{\gamma}_{j_t} = s^{\beta}_t+n$.
This shows that the outcomes of all shifted shuffles of two parking functions can be determined only from the outcomes of those two parking functions. 

The equivalence classes of $\PF_n$ under $\out$ are easily seen to be in bijection with $\mathfrak{S}_n$.  The proof of above shows that the outcomes of the shifted shuffles of two parking functions correspond exactly to the shifted shuffles of the outcomes. Thus, we have the structure of $\FQSym$.
\end{proof}

We require additional background on $\QSym$ for our next statistic. Recall that $\QSym$ is isomorphic to the shuffle algebra $\Sh$, and thus has various \emph{shuffle bases} of the form $\{X_{\delta}\}_{\delta}$ indexed by compositions $\delta$ with the product rule $X_{\delta}X_{\eta} = \sum_{\iota \in \delta \shuffle \eta} X_{\iota}$, as shown in \cite{liu2023shufflebasesquasisymmetricpower}. Our next result shows that the equivalence algebra of the displacement statistic is isomorphic to a subalgebra of $\QSym$ spanned by shuffle basis elements indexed by \emph{subexcedant words}.

\begin{theorem}\label{thm:Disp}(Weak shuffle-compatibility of the displacement sequence)
\begin{enumerate}
    \item The displacement sequence statistic $\Disp$ is weakly shuffle-compatible on parking functions.
    \item The linear map on $\mathcal{A}^{\PF}_{\Disp}$ defined by \[[\alpha]_{\Disp} \rightarrow X_{\Disp(\alpha)[1]},\] is an algebra isomorphism between $\mathcal{A}^{\PF}_{\Disp}$ and the subalgebra of $\QSym$ spanned by basis elements $X_{\delta}$ where $\delta = (\delta_1, \delta_2, \ldots, \delta_k) \in W_{\N}$ has the property that $\delta_{i}\leq i$.
    \end{enumerate}
\end{theorem}

\begin{proof} 
First, consider $\alpha \in \PF_n$ and $\beta \in \PF_m$. Observe that $\{\{ \Disp(\gamma) : \gamma \in \alpha \shifted \beta \}\}$ can be determined by $\Disp(\alpha)$, $\Disp(\beta)$, $|\alpha|$, and $|\beta|$ since it is simply equal to $\Disp(\alpha) \shuffle \Disp(\beta)$.

Note throughout the rest of this proof, we refer to integer compositions and words in $\N$ interchangeably.
Now, consider the map $\varsigma:\mathcal{A}^{\PF}_{\Disp} \rightarrow \QSym$ defined by $\varsigma([\alpha]_{\Disp})=X_{\Disp(\alpha)[1]}$.  The map $\varsigma$ is clearly injective, so we next check that it is a homomorphism of algebras.   Observe that \[\varsigma([\alpha]_{\Disp})\varsigma([\beta]_{\Disp})=X_{\Disp(\alpha)[1]}X_{\Disp(\beta)[1]} = \sum_{\tau \in \Disp(\alpha)[1]\shuffle \Disp(\beta)[1]}X_{\tau}.\]  Similarly, we have \[\varsigma([\alpha]_{\Disp}[\beta]_{\Disp}) = \varsigma(\sum_{\gamma \in \alpha \shifted \beta} [\gamma]_{\Disp}) = \sum_{\gamma \in \alpha \shifted \beta} X_{\Disp(\gamma)[1]}.\] Following the first part of the proof, if $\gamma \in \alpha \shifted \beta$ then $\Disp(\gamma)$ is the corresponding shuffle of $\Disp(\alpha)$ and $\Disp(\beta)$. As a result, the two expressions above are equivalent.
It remains to show that the image of $\varsigma$ is the specified subalgebra. Let $\alpha$ be a parking function. When car $i$ parks, it can be displaced by at most $i-1$ spots since only $i-1$ cars will already be parked. Thus, if $\delta = \Disp(\alpha)[1]$, we have $\delta_i \leq i$. Finally, we check that $X_{\delta}$ is in the image of $\varsigma$ for any composition $\delta$ with $\delta_i \leq i$.  Given such a $\delta$, we want to construct a parking function $\alpha$ where $\Disp(\alpha)=\delta[-1]$. To do so, let $\alpha_i = i - (\delta_i-1)$.  In this construction, it is clear that $\alpha_i \leq i$ so $\alpha$ is certainly a parking function, specifically one with the outcome $(1,2,\ldots, k)$. Thus, we have shown that the image of $\varsigma$ is exactly the subalgebra specified.
\end{proof}

It follows from the first part of the proof above, and the fact that $\disp(\gamma) = \sum_{i\in \Disp(\gamma)}i$,  that if $\gamma \in \alpha \shifted \beta$, then $\disp(\gamma) = \disp(\alpha)+\disp(\beta)$. This yields the following results on the displacement statistic.

\begin{corollary}\label{cor:disp}(Weak shuffle-compatibility of the displacement)
\begin{enumerate}
    \item The displacement statistic $\disp$ is weakly shuffle-compatible on parking functions.
    \item The linear map on $\mathcal{A}^{\PF}_{\disp}$ defined by $[\alpha]_{\disp} \mapsto {\frac{1}{|\alpha|!}}q^{\disp(\alpha)}x^{|\alpha|}$ is an algebra isomorphism between $\mathcal{A}^{\PF}_{\disp}$ and the span of $\{1\} \cup \{q^jx^n\}_{1 \leq n , 0 \leq j \leq \binom{n}{2}}$, a subalgebra of $\mathbb{Q}[q,x]$.
\end{enumerate}
\end{corollary}

Additionally, from the proof of Theorem \ref{thm:Disp},  for $\gamma \in \alpha \shifted \beta$, we have $\mathsf{maxDisp}(\gamma) = \max(\mathsf{maxDisp}(\alpha),\mathsf{maxDisp}(\beta)).$ This yields the following results on the statistic for the maximum displacement of a single car.

\begin{corollary} The maximum displacement of a single car statistic $\mathsf{maxDisp}$ is weakly shuffle-compatible on parking functions.
\end{corollary}

In \cite{novelli2016binary}, the authors prove the existence of a \emph{binary shuffle basis} $\{Z_{\delta}\}_{\delta}$ (indexed by compositions) of $\QSym$ with the following product rule. Given a composition $\delta$, let $\epsilon(\delta)$ be the binary word where $\epsilon(\delta)_n = 1$ and $\epsilon(\delta)_i=1$ if $i \in \Set(\delta)$ and $\epsilon(\delta)_i=0$ if $i \not\in \Set(\delta)$ for $1 \leq i < n$. For a binary word $w$ ending in $1$, let $\epsilon^{-1}(w)=\delta$ where $\epsilon(\delta)=w$.
Then, $Z_{\delta}Z_{\eta} = \sum_{w \in \epsilon(\delta)\shuffle \epsilon(\eta)}Z_{\epsilon^{-1}(w)}$. 

We introduce the following notation to translate between lucky sets and binary words ending in $1$. Given a parking function $\alpha$, let $LW(\alpha)$ be the word with $LW_i(\alpha) = 1$ if car $i$ is lucky and $LW_i(\alpha) = 0$ if car $i$ is not lucky.  
Let $LW^r(\alpha)$ be the reverse of the word $LW(\alpha)$.
Note that $LW_1(\alpha)=1$ always holds, so $LW^r(\alpha)$ is a binary word ending in $1$. Observe that $LW(\alpha)$ only depends on $\Lucky(\alpha)$, hence the map is well-defined on $\Lucky$-equivalence classes. Thus, this map induces a bijection between Lucky car sets and binary words ending in 1.

\begin{theorem}\label{thm:Lucky}(Shuffle-compatibility of the lucky car set)
\begin{enumerate}
    \item The lucky car set statistic $\Lucky$ is weakly shuffle-compatible on parking functions.
    \item The linear map on $\mathcal{A}_{\Lucky}^{\PF}$ defined by \[[\alpha]_{\Lucky} \mapsto Z_{\epsilon^{-1}(LW^r(\alpha))},\] is an algebra isomorphism beweetn $\mathcal{A}^{\PF}_{\Lucky}$ and $\QSym$.
\end{enumerate}    
\end{theorem}

\begin{proof}
    Consider some $\alpha \in \PF_n$ and $\beta \in \PF_m$. Consider $\gamma \in \alpha \shifted \beta$ where indices $\{i_1 < \cdots <  i_n\}$ correspond to entries originally coming from $\alpha$ and indices $\{j_1 < \cdots < j_m\}$ correspond to entries originally coming from $\beta$. Following what we know about the outcome from the previous proof, any lucky car in $\alpha$ will be lucky again in $\gamma$, as will any lucky car from $\beta$. Thus, if $\Lucky(\alpha) = \{l^{\alpha}_1, \ldots, l^{\alpha}_{x}\}$ and $\Lucky(\beta) = \{l_{1}^{\beta}, \ldots, l^{\beta}_{y}\}$ then \[\Lucky(\gamma) = \{i_{l^{\alpha}_1}, \ldots, i_{l^{\alpha}_x}, j_{l_{1}^{\beta}}, \ldots, j_{l_{y}^{\beta}}\}.\] So the lucky sets (and as a result, the unlucky sets) of all shuffles of two parking functions can be determined only from their lucky (resp. unlucky) sets. 

Consider the map $\varsigma:\mathcal{A}^{\PF}_{\Lucky} \rightarrow \QSym$ defined by $[\alpha]_{\Lucky} \mapsto Z_{\epsilon^{-1}(LW^r(\alpha))}$. This map is a bijection between the $\Lucky$-equivalence classes and $\{Z_{\delta}\}$ basis elements given that $LW^r$ bijects between lucky car sets and binary words ending in $1$ and $\epsilon^{-1}$ bijects between binary words ending in $1$ and integer compositions.  We see that \[\varsigma([\alpha]_{\Lucky})\varsigma([\beta]_{\Lucky}) = Z_{\epsilon^{-1}(LW^r(\alpha))}Z_{\epsilon^{-1}(LW^r(\beta))} = \sum_{w\in LW^r(\alpha)\shuffle LW^r(\beta)}Z_{\epsilon^{-1}(w)},\] and also,
\[\varsigma([\alpha]_{\Lucky}[\beta]_{\Lucky}) = \varsigma(\sum_{\gamma \in \alpha \shifted \beta} [\gamma]_{\Lucky}) = \sum_{\gamma \in \alpha \shifted \beta} Z_{\epsilon^{-1}(LW^r(\gamma))}.\] Note that $LW^r$ bijects between shifted shuffles of $\alpha$ and $\beta$ to shuffles of $LW^r(\alpha)$ and $LW^r(\beta)$, and so the two expressions above are equal. Thus, $\varsigma$ is an algebra isomorphism.\end{proof}

\begin{proposition}
    The $\Lucky$ shifted shuffle algebra $\mathcal{A}^{\PF}_{\Lucky}$ is a quotient of the $\Disp$ shifted shuffle algebra $\mathcal{A}_{\Disp}^{\PF}$.
\end{proposition}

\begin{proof}
    Let $\phi: \mathcal{A}^{\PF}_{\Disp} \rightarrow \mathcal{A}^{\PF}_{\Lucky}$ be defined by $[\alpha]_{\Disp} \mapsto [\alpha]_{\Lucky}$, noting that $\Lucky(\alpha) = \{i : \Disp(\alpha)_i=0\}$. There is some displacement sequence for every lucky car set, and so the map is surjective. Observe that \[\phi([\alpha]_{\Disp})\phi([\beta]_{\Disp}) = [\alpha]_{\Lucky}[\beta]_{\Lucky} = \sum_{\gamma \in \alpha \shifted \beta} [\gamma]_{\Lucky} = \phi(\sum_{\gamma \in \alpha \shifted \beta} [\gamma]_{\Disp}) = \phi([\alpha]_{\Disp}[\beta]_{\Disp}).  \]  Then $\phi$  a surjective morphism, so $\mathcal{A}^{\PF}_{\Lucky}$ is isomorphic to $\mathcal{A}^{\PF}_{\Disp}/\ker(\phi)$.
\end{proof}

Results on the lucky car number follow from our work on the lucky car set.

\begin{corollary}(Weak shuffle-compatibility of the lucky car number)
\begin{enumerate}
    \item The lucky car number statistic $\lucky$ is weakly shuffle-compatible on parking functions.
    \item The linear map on $\mathcal{A}^{\PF}_{\lucky}$ defined by \[[\alpha]_{\lucky} \mapsto \frac{1}{|\alpha|!}q^{\lucky(\alpha)}x^{|\alpha|},\] is an algebra isomorphism from $\mathcal{A}^{\PF}_{\lucky}$ to the span of $\{1\} \cup \{q^jx^n\}_{1 \leq n, 1\leq j \leq n}$, a subalgebra of $\mathbb{Q}[q,x]$.
\end{enumerate}
\end{corollary}

Since the lucky car set and number entirely determine the unlucky car set and number (specifically, $[\alpha]_{\Lucky} = [\alpha]_{\Unlucky}$ and $[\alpha]_{\lucky} = [\alpha]_{\unlucky}$ for all $\alpha \in \PF$), we can extend our results to these statistics.

\begin{corollary} (Weak shuffle-compatibility of the unlucky car set and number)
\begin{enumerate}
    \item The unlucky car set and number statistics $\Unlucky$ and $\unlucky$ are weakly shuffle-compatible on parking functions.
    \item The $\Unlucky$ shifted shuffle algebra $\mathcal{A}^{\PF}_{\Unlucky}$ and the $\unlucky$ shifted shuffle algebra are isomorphic to the $\Lucky$ shifted shuffle algebra $\mathcal{A}^{\PF}_{\Lucky}$ and the $\lucky$ shifted shuffle algebra  $\mathcal{A}^{\PF}_{\lucky}$, respectively.
\end{enumerate}
    
\end{corollary}

 Our next statistic takes advantage of the subclass of parking functions that are Catalan objects. Recall the definitions of the $\mathbf{P}$-basis of $\CQSym$ from the beginning of the section.

\begin{theorem} (Weak shuffle-compatibility of the weakly increasing rearrangement)
\begin{enumerate}
    \item The weakly increasing rearrangement statistic $\winc$ is weakly shuffle-compatible on parking functions.
    \item The linear map on $\mathcal{A}_{\winc}^{\PF}$ defined by \[[\alpha]_{\winc} \mapsto \frac{1}{|\alpha|!}\mathbf{P}_{\winc(\alpha)},\] is an algebra isomorphism between $\mathcal{A}_{\winc}^{\PF}$ and $\CQSym$.
    \end{enumerate}
\end{theorem}

\begin{proof} Let $\alpha \in \PF_n$ and $\beta \in \PF_m$. Observe that if $\gamma \in \alpha \shifted \beta$ then $\winc(\gamma) = \winc(\alpha)\cdot(\winc(\beta))[n] = \winc(\alpha \cdot \beta[n])$, so weak shuffle-compatibility follows.

Now, consider the map $\varsigma: \mathcal{A}^{\PF}_{\winc} \rightarrow \CQSym $ defined by $[\alpha]_{\winc} \mapsto \frac{1}{n!}\mathbf{P}_{\winc(\alpha)}$. Weakly increasing parking functions index both $\winc$-equivalence classes and basis elements of $\CQSym$ so this map is a bijection. 
    Observe that \[\varsigma([\alpha]_{\winc})\varsigma([\beta]_{\winc}) = (\frac{1}{n!}\mathbf{P}_{\winc(\alpha)})(\frac{1}{m!}\mathbf{P}_{\winc(\beta)}) = \frac{1}{n!m!} \mathbf{P}_{\winc(\alpha) \cdot \winc(\beta)[n]}.\] Similarly, 
    \[\varsigma([\alpha]_{\winc}[\beta]_{\winc}) = \varsigma(\sum_{\gamma \in \alpha \shifted \beta} [\gamma]_{\winc}) = \varsigma(\binom{n+m}{n}[\alpha \cdot \beta[n]]_{\winc}) = \frac{1}{(n+m)!} \binom{n+m}{n}\mathbf{P}_{\winc(\alpha \cdot \beta[n])}.\]
    These two expressions are equivalent given that $\winc(\alpha)\cdot\winc(\beta)[n] = \winc(\alpha \cdot \beta[n])$, and so we have that $\varsigma$ is an isomorphism.  
\end{proof}

We conclude this section with five more weakly shuffle-compatible statistics on parking functions. 

\begin{proposition} \label{prop:lel} The number of leading elements $\lel$ is weakly shuffle-compatible on parking functions.
\end{proposition}

\begin{proof}
    Consider $\alpha \in \PF_n$ and $\beta \in \PF_m$ where $\lel(\alpha)=k_{\alpha}$ and $\lel(\beta) = k_{\beta}$.  Consider $\gamma \in \alpha \shifted \beta$. If $\gamma_1 = \alpha_1$ then $\lel(\gamma)=k_{\alpha}$, but if $\gamma_1 = \beta_1$ then $\lel(\gamma)=k_{\beta}$. Thus, in $\{\{ \lel(\gamma) :  \gamma \in \alpha \shifted \beta\}\}$ there will be $\binom{n+m-1}{n-1}$ instances of $k_{\alpha}$ and $\binom{n+m-1}{m-1}$ instances of $k_{\beta}$.
\end{proof}

\begin{proposition}
    The number of ones $\ones$ is weakly shuffle-compatible on parking functions.  
\end{proposition}

\begin{proof}
    Observe that if $\gamma \in \alpha \shifted \beta$ then $\ones(\gamma) = \ones(\alpha)$.
\end{proof}

\begin{proposition}
    The size of the initial strictly increasing segment statistic $\isi$ is weakly shuffle-compatible on parking functions.
\end{proposition}

\begin{proof}
   Consider $\alpha \in \PF_n$ and $\beta \in \PF_m$ where $\isi(\alpha)=k_{\alpha}$ and $\isi(\beta) = k_{\beta}$. Let $\gamma$ be a shifted shuffle of $\alpha$ and $\beta$ corresponding to the sets $X = \{x_1 < \cdots < x_n\} \sqcup Y = \{y_1 < \cdots < y_m\} = [n+m]$ where $\gamma_{x_i} = \alpha_i$ and $\gamma_{y_i}=\beta_i$.  Let $j$ be the largest number such that $(x_1, x_2, \ldots, x_j) = (1,2, \ldots, j)$, or $0$ if such a number does not exist, and let $h$ be the largest number such that $y_2 = y_1 + 1, y_3 = y_2+1, \ldots, y_h = y_{h-1}+1$,  or $0$ if such a number does not exist.  Then, 
\[\isi(\gamma) = \begin{cases}
    j+\min\{h,k_{\beta}\} & \text{if } j\leq k_\alpha\\  
    k_{\alpha} & \text{if } j > k_{\alpha}
\end{cases}.\]
It follows that we can determine $\{\{ \isi(\gamma) : \gamma \in \alpha \shifted \beta \}\}$ completely from $\isi(\alpha), \isi(\beta), |\alpha|,$ and $|\beta|$.
\end{proof}

\begin{proposition}
        The sum and sum minus length statistics, $\sumst$ and $\mathsf{sml}$, are weakly shuffle-compatible on parking functions.
\end{proposition}

\begin{proof} Let $\alpha \in \PF_n$ and $\beta \in \PF_m$.
    Observe that if $\gamma \in \alpha \shifted \beta$, then $\sumst(\gamma) = \sumst(\alpha) + \sumst(\beta) + mn$, as each entry of $\beta$ is increased by $n$ before being shuffled with $\alpha$. Similarly, $\mathsf{sml}(\gamma) = \mathsf{sml}(\alpha) + \mathsf{sml}(\beta) + mn$.
\end{proof}

See Appendix \ref{ap:pf} for examples of statistics that are not weakly shuffle-compatible on parking functions.

\section{Set Partitions}\label{sec:setpartitions}

A \emph{set partition} $\pi = B_1 / B_2/ \cdots / B_k$ of a set $I$ is an unordered collection of subsets $B_1, \ldots, B_k \subseteq I$ such that $B_i \cap B_j  = \emptyset$ for all $i,j \in [k]$ with $i \not= j$, and $B_1 \cup B_2 \cup \cdots \cup B_k = I$. We write $\pi \vdash I$ to denote that $\pi$ partitions $I$.  Let $|\pi|$ denote the {\em size} of $\pi$, which is number of elements in $\pi$, or $|I|$. Let $\Psi_n$ denote the collection of set partitions of sets $I$ where $I \subseteq \N, |I|=n$, and let $\Psi = \cup_{n \geq 0} \Psi_n$. Given a set partition $\pi$ of $I = \{i_1 < i_2 < \cdots <i_n\} \subseteq \N$, we let the \emph{standardization} of $\pi$, denoted $\std(\pi)$, be the set partition of $[n]$ obtained by replacing $i_j$ with $j$ for each $j\in[n]$. Let $\Pi_n$ denote set partitions of $[n]$ and let $\Pi = \cup_{n \geq 0} \Pi_n$. 

We can visualize set partitions in $\Pi$ using \emph{arc diagrams}. The diagrams consist of vertices labeled by elements in $I$ from least to greatest with arcs connecting each pair of elements that share a block and have no other elements in the same block in between them. 

 Table \ref{tab:setP} lists various statistics on set partitions that appear in the literature or on FindStat, all of which we show to be (weakly) shuffle-compatible. For these definitions, let $\pi$ be a set partition of $I = \{i_1< i_2 < \cdots < i_n\} \subseteq \mathbb{N}$. 

\begin{table}[h!]
    \centering
    \begin{tabular}{|c|c|c|} \hline
    \textbf{Name} & \textbf{Definition} & \textbf{FindStat} \\ \hline
     Block sizes & $\Bk(\pi) = \{\{|B_1|, |B_2|, \ldots, |B_k| \}\}$ & \\ \hline
        Block number & $\bk(\pi)=|\Bk(\pi)|$ & St000105\\ \hline
        Rank & $\rank(\pi) = |\pi|-\bk(\pi)$  &  St000211\\ \hline
        Succession set & $\Succ(\pi) = \{i : i, i+1 \in B_j\}$ & \\ \hline
        Succession number & $\suc(\pi) = |\Succ(\pi)|$ & St000502\\ \hline
       
        Set of singletons & $\Sg(\pi) = \{i : B_j = \{i\} \text{ for some }j\}$ & \\ \hline
         Number of singletons & $\sg(\pi) = |\Sg(\pi)|$ &  St000247\\ \hline
         Set of openers & $\Op(\pi) = \{i : i = \min(B_j) \text{ for some }j\}$ & \\ \hline
         Set of closers & $\Cl(\pi) = \{i : i = \max(B_j) \text{ for some }j\}$ & \\ \hline
         Set of transients & $\Tr(\pi) = \{i : i \in B_j, \min(B_j)<i<\max(B_j) \text{ for some }j\}$ & \\ \hline
         Number of nonsingleton blocks & $\mathsf{nsg}(\pi) = \bk(\pi) - \sg(\pi)$ & St000251 \\ \hline
         Size of the block containing 1 & $\bk_1(\pi) = |B_i|$, $1\in B_i$  & St000504 \\ \hline
          Closer of the block containing 1 & $\mathsf{cl_1(\pi) = \max(B_i) }$ where $1 \in B_i$ & St000505  \\ \hline
         Largest opener & $\mathsf{lop}(\pi)=\max(\Op(\pi))$ & St000728 \\ \hline
         Smallest closer & $\mathsf{scl}(\pi)=\min(\Cl(\pi))$ & St000971 \\ \hline
         Number of terminal closers & $\mathsf{tcl}(\pi)= |\{j\in [n] : i_j, i_{j+1}, \ldots, i_n \in \mathsf{Cl}(\pi)\}|$  & St001050 \\ \hline
         Maximal size of a block & $\mathsf{maxBk}(\pi) = \max(\Bk(\pi))$ & St001062 \\ \hline
         Minimal size of a block & $\mathsf{minBk}(\pi)=\min(\Bk(\pi))$ & St001075 \\ \hline
        
Number of occ. of pattern 12 & $\occ_{12}(\pi) = |\{(i,j) : i<j, \exists h, i,j \in B_h\}|$ & St000558 \\ \hline
Number of occ. of pattern 123 &  $\occ_{123}(\pi) = |\{(i,j,k) : i < j <k, \exists h, i,j,k \in B_h\}|$ & St000561 \\ \hline
Number of occ. of pattern  1/2 & $\occ_{1/2}(\pi) = |\{(i,j) : i<j, \exists h, i\in B_h, j\not\in B_h\}|$ & St000564  \\ \hline     
    \end{tabular}
    \caption{Statistics on set partitions \cite{ kasraoui2006distribution, FindStat} }
    \label{tab:setP}
\end{table}

Let $I=\{i_1<i_2<\cdots < i_n\},S =\{s_1<s_2<\cdots< s_n\} \subset \N$, and $\pi \vdash I$. Then, $\shift_S(\pi)$ is the set partition of $S$ obtained by replacing each element $i_j$ with $s_j$ for each $j \in [n]$. Given two disjoint set partitions $\tau \vdash I$ and $\theta \vdash J$, define the \emph{arc-shuffle} of $\tau$ and $\theta$ by 
\[\tau \mix \theta = \{\{\shift_S(\tau)\cup \shift_{S^c}(\theta) : S \subseteq I \cup J, |S|=|I|, S^c = (I\cup J) \setminus S\}\}.\] For two set partitions $\rho \vdash [n]$ and $\nu \vdash [m]$, define the \emph{shifted arc-shuffle} of $\rho$ and $\nu$ to be 
\[\rho \smix \nu = \rho \mix \nu[n],\] where $\nu[n]$ denotes that $n$ should be added to each element of $\nu$. This is equivalent to $\rho \smix \nu = \{\{ \shift_S(\rho)\cup \shift_{S^c}(\nu) : S \subseteq [n+m], |S|=n, S^c = [n+m]\setminus S \}\}$.  We refer to this as the arc-shuffle because it is equivalent to shuffling the arc diagrams of the set partitions while leaving the node labels fixed in place.

The Hopf algebra structure on set partitions is given by the symmetric functions in noncommuting variables, which were introduced in \cite{Ros06} and shown to be a Hopf algebra, called $\NCSym$, in \cite{Ber08}. We consider the dual algebra, $\NCSym^*$, which has a basis $\{\mathbf{w}_{\pi}\}_{\pi \in \Pi}$. This basis has the product rule given by 
\[ \mathbf{w}_{\tau}\mathbf{w}_{\theta} = \sum_{\pi \in \tau \smix \theta} \mathbf{w}_{\pi}.\] 

Note that unlike the commutativity of the shuffle on words, but noncommutativity of the shifted shuffle on permutations and parking functions, the shuffle and shifted shuffle on set partitions are both commutative.

\subsection{Weak-shuffle-compatibility on set partitions}

We begin by considering statistics that are weakly shuffle-compatible on $\Pi$ and their algebraic structures.

\begin{definition}
A statistic $\St$ is weakly shuffle-compatible on set partitions if $\{\{\St(\pi) : \pi \in \tau \smix \theta\}\}$ is determined by $\St(\tau)$, $\St(\theta)$, $|\tau|$, and $|\theta|$ for any two set partitions $\tau, \theta \in \Pi$.
\end{definition}

Equivalently, $\St$ is weakly shuffle-compatible on set partitions if \[ \{\{ \St(\pi) : \pi \in \tau \smix \theta \}\} = \{\{ \St(\pi) : \rho \smix \nu \}\}\]
for any two pairs of set partitions $\{\tau, \theta\}$ and $\{ \rho, \nu \}$ from $\Pi$ such that $|\tau|=|\rho|$, $|\theta|=|\nu|$, $\St(\tau)=\St(\rho)$, and $\St(\theta) = \St(\nu)$.

We denote the $\St$-equivalence class of $\pi \in \Pi$ by $[\pi]_{\St}^{\Pi} = \{ \tau \in \Pi : \St(\pi)=\St(\tau), |\pi|=|\tau|\}$, where we may drop the $\Pi$ superscript when it is clear from context. Let $\Pi/\St$ denote the set of all such $\St$-equivalence classes.

\begin{definition}
   Let $\St$ be weakly shuffle-compatible on set partitions. The $\St$ shifted shuffle algebra, denoted $\mathcal{A}^{\Pi}_{\St}$, is defined as the algebra with basis elements $[\pi]_{\St}^{\Pi}\in \Pi/\St$ and the product 
    \[[\tau]_{\St}[\theta]_{\St} = \sum_{\pi \in \tau \smix \theta} [\pi]_{\St}.\]
\end{definition}

\begin{proposition}\label{prop:setparquo}
If $\St$ is a weakly shuffle-compatible statistic on set partitions, then $\mathcal{A}^{\Pi}_{\St}$ is a quotient algebra of $\NCSym^*$. 
\end{proposition}

\begin{proof}
     Let $\St$ be a weakly shuffle-compatible statistic on set partitions. Consider the surjective map $\varsigma: \NCSym^* \rightarrow \mathcal{A}^{\Pi}_{\St}$ defined by $\mathbf{w}_{\pi} \rightarrow [\pi]_{\St}^{\Pi}$. Additionally, \[\varsigma(\mathbf{w}_{\tau}\mathbf{w}_{\theta}) = \varsigma(\sum_{\pi \in \tau \shifted \theta} \mathbf{w}_{\pi})= \sum_{\pi \in \tau \shifted \theta} [\pi]^{\Pi}_{\St} = [\tau]^{\Pi}_{\St}[\theta]^{\Pi}_{\St}=\varsigma(\mathbf{w}_{\tau})\varsigma(\mathbf{w}_{\theta}).\] Since $\varsigma$ is a surjective morphism, we have that $\mathcal{A}^{\Pi}_{\St}$ is isomorphic to $\NCSym^*/\ker(\varsigma)$.
\end{proof}

For the algebraic structure of the succession set equivalence classes, recall the tie basis $\{T_{\alpha}\}_\alpha$ of $\QSym$ from Theorem \ref{thm:T_basis}.

\begin{theorem}\label{thm:Succ_WSC}(Weak shuffle-compatibility of the succession set)
\begin{enumerate}
    \item The succession set statistic $\Succ$ is weakly shuffle-compatible on set partitions.
    \item The linear map on $\mathcal{A}^{\Pi}_{\Succ}$ defined by \[[\pi]_{\Succ} \mapsto T_{\Comp(\Succ(\pi))}\] is an algebra isomorphism between $\mathcal{A}_{\Succ}^{\Pi}$ and $\QSym$.
\end{enumerate}
\end{theorem}
\begin{proof} Given $\tau \in \Pi_n, \theta \in \Pi_m$, each set partition in $\tau \smix \theta$ can be associated with a unique set $S \subseteq [n+m]$ with $|S| = n$. Suppose $\Succ(\tau)=\{t_1, t_2, \ldots, t_{k_{\tau}}\}$ and $\Succ(\theta) =\{u_1, u_2, \ldots, u_{k_{\theta}}\}$.  Suppose $S = \{s_1 < s_2 < \cdots < s_n \}$ and $[n+m]\setminus S = \{r_1, r_2, \ldots, r_m\}$.  Then, for the partition $\pi \in \tau \smix \theta$ given by $\pi = \shift_S(\tau) \cup \shift_{[n+m]\setminus S}(\theta)$, we have 
that $i \in \Succ(\pi)$ if $i=s_{t_j}$ and $i+1=s_{t_j+1}$ or if $i=r_{u_j}$ and $i+1 = r_{u_j+1}$ for some $j$. Thus, we can determine the succession set for each shuffle of $\tau$ and $\theta$ based only on $\Succ(\tau), \Succ(\theta), |\tau|,$ and $|\theta|$.

Let $\varrho: \mathcal{A}_{\Succ}^{\Pi} \rightarrow \QSym$ be defined by $\varrho([\pi]_{\Succ})=T_{\Comp(\Succ(\pi))}$.  This map is bijective as we are mapping between basis elements corresponding to all subsets of $[n-1]$ or (compositions of $n$) on each side. Then, \[\varrho([\tau]_{\Succ}[\theta]_{\Succ}) = \varrho(\sum_{\pi \in \tau \smix \theta} [\pi]_{\Succ}) = \sum_{\pi \in \tau \smix \theta} T_{\Comp(\Succ(\pi))}.\] 
Similary, \[\varrho([\tau]_{\Succ})\varrho([\theta]_{\Succ}) = T_{\Comp(\Succ(\tau))}T_{\Comp(\Succ(\theta))} = \sum_{v \in w^{\Tie}_{\Comp(\Succ(\tau))}\shifted w^{\Tie}_{\Comp(\Succ(\theta))}} T_{\Comp(\Tie(v))}.\] 
Note that the description of the succession set for a specific shuffle earlier in the proof is identical to the description of the tie set for a specific shuffle given in the proof of Theorem~\ref{thm:Tie}. As such, we can construct a bijection sending the partition $ \pi \in \tau \smix \theta$ associated to a specific set $S$ as above to the word $v \in w^{\Tie}_{\Comp(\Succ(\tau))} \shifted w^{\Tie}_{\Comp(\Succ(\theta))}$ where the indices of $v$ filled with the letters of $w^{\Tie}_{\Comp(\Succ(\tau))}$ are given by $S$.  Here, $\Succ(\pi)=\Tie(v)$, making the two expressions above equivalent.
\end{proof}

\begin{theorem} 
(Weak shuffle compatibility of the succession number)
\begin{enumerate}
    \item The succession number statistic $\suc$ is weakly shuffle-compatible on set partitions.
    \item The linear map on $\mathcal{A}^{\Pi}_{\suc}$ defined by \[[\pi]^{\Pi}_{\suc} \mapsto [f(\pi)]^W_{\tie},\] where $f(\pi) = (1)^{\suc(\pi)+1}\cdot(\suc(\pi)+2, \ldots, |\pi|-1, |\pi|)$, is an algebra isomorphism from $\mathcal{A}_{\suc}^{\Pi}$ to $\mathcal{A}^W_{\tie}$.
\end{enumerate}
\end{theorem}
\begin{proof}
    By the logic in the proof of Theorem \ref{thm:Succ_WSC} above, one can instead construct an isomorphism $\varsigma:\mathcal{A}_{\Succ}^{\Pi}\rightarrow \mathcal{A}_{\Tie}^{W}$ where $\varsigma([\pi]^{\Pi}_{\Succ})=[\alpha]_{\Tie}^{W}$ where $[\alpha]^W_{\Tie} = \{ \beta \in W : \Tie(\beta)=\Succ(\pi), |\beta|=|\pi|\}$.  Given that $\suc(\pi) = |\Succ(\pi)|$ and $\tie(\alpha) = |\Tie(\alpha)|$, the shuffle-compatibility of $\tie$ (Theorem \ref{thm:tie}) implies the shuffle-compatibility of $\suc$.
    The map $\varsigma^{-1}$ projects down to an isomorphism $\vartheta^{-1}: \mathcal{A}_{\tie}^W \rightarrow \mathcal{A}_{\suc}^{\Pi}$ where $\vartheta: \mathcal{A}_{\suc}^{\Pi} \rightarrow \mathcal{A}_{\tie}^W$ is given by $\vartheta([\pi]_{\suc}^{\Pi}) =[\alpha]_{\tie}^W$ and $[\alpha]_{\tie}^W = \{ \beta \in W : \tie(\beta)=\suc(\pi), |\beta|=|\pi| \} = [f(\pi)]_{\tie}^W$. 
\end{proof}

 Let $W(0,1)$ denote the set of binary words and let $\mathbb{Q}_{\shuffle}\langle 0,1 \rangle$ denote the shuffle algebra on binary words~\cite{novelli2016binary} with a basis $\{W_u\}_{u \in W(0,1)}$. The product on this basis is given by $W_{v}W_{w} = \sum_{u \in v \shuffle w} W_u$. Given a set $S$, let $bw(S)$ be the binary word of length $|S|$ where the $i^{\text{th}}$ entry is 1 if and only if $i \in S$.

\begin{theorem}\label{thm:sg_set}(Weak shuffle-compatibility of the singletons set)
\begin{enumerate}
    \item The set of singletons statistic $\Sg$ is weakly shuffle-compatible on set partitions. 
    \item The linear map on $\mathcal{A}_{\Sg}^{\Pi}$ defined by \[[\pi]_{\St} \mapsto W_{bw(\Sg(\pi))},\] is an algebra isomorphism between $\mathcal{A}_{\Sg}^{\Pi}$ and
    the subalgebra of $\mathbb{Q}_{\shuffle}\langle 0,1 \rangle$ spanned by \[\{W_u : u \in W(0,1), \sum u_i \not= |u|-1\}.\]
    \end{enumerate}
\end{theorem}

\begin{proof}

For $\tau, \theta \in \Pi$, let $S \subseteq [|\tau|+|\pi|]$ with $|S|=|\tau|$, and let $\pi = \shift_{S}(\tau) \cup \shift_{[|\tau|+|\theta|] \setminus S}(\theta)$.   Then $\Sg(\pi) = \shift_S(\Sg(\tau)) \cup \shift_{[|\tau|+|\theta|] \setminus S}(\Sg(\theta))$. Thus, the $\Sg$ value for the set of shifted arc-shuffles of $\tau$ and $\theta$ can be determiend only from $\Sg(\tau), \Sg(\theta), |\tau|, $ and $|\theta|$.

Let $\varsigma$ be the map defined by $[\pi]_{\Sg} \mapsto W_{bw(\Sg(\pi))}$. The possible values of $\Sg(\pi)$ for $\pi 
\in \Pi_n$ are given by exactly the subsets $[n]$ that do not have exactly $n-1$ elements. To see this, consider that to achieve a specific subset of $[n]$ as a singleton set, we make all the elements of the set singletons and put the remaining elements in a block all together. A subset of size $n-1$ cannot be achieved as a singleton set, as it is impossible to have one non-singleton element. As such, the map $\varsigma$ is bijective.
Further, 
\[\varsigma([\tau]_{\St}[\theta]_{\St})=\varsigma(\sum_{\pi \in \tau \smix \theta} [\pi]_{\Sg}) = \sum_{\pi \in \tau \smix \theta} W_{bw(\Sg(\pi))}.\] Observe that the shuffles of $\tau$ and $\theta$ correspond to translating the values $\tau$ to a subset of $|\tau|$ values of $[|\tau|+|\theta|]$ and the values of $\theta$ to the remaining values of $[|\tau|+|\theta|]$. The values that are singletons are simply translated under this shift for each arc-shuffle. Similarly, when shuffling binary words, the $|\tau|$ entries from the first word are translated to some $|\tau|$ subset of the indices $[|\tau|+|\theta|]$ and the entries of $\theta$ fill in the remaining slots. Each $1$ is translated along with the rest of the word. Matching up these 1's with singleton values yields 
\[\sum_{\pi \in \tau \smix \theta} W_{bw(\Sg(\pi))}= \sum_{u \in bw(\Sg(\tau)) \shuffle bw(\Sg(\theta))}W_u= W_{bw(\Sg(\tau))}W_{bw(\Sg(\theta))} = \varsigma([\tau]_{\Sg})\varsigma([\theta]_{\Sg}).\] 
Thus, $\varsigma$ is an isomorphism between $\mathcal{A}_{\Sg}^{\Pi}$ and the specified subalgebra of $\mathbb{Q}_{\shuffle}\langle 0,1 \rangle$.
\end{proof}

In the next three proofs, we return to the binary shuffle basis $\{Z_{\delta}\}_{\delta}$ of $\QSym$. Tracking the 1's in shuffles of binary words is equivalent to the tracking of certain elements when passed through the arc-shuffle of set partitions, and we form various bijections around this idea.

\begin{proposition}\label{prop:tr} (Weak shuffle-compatibility of the transient set)
\begin{enumerate}
    \item The set of transients statistic $\Tr$ is weakly shuffle-compatible on set partitions.
    \item The linear map on $\mathcal{A}^{\Pi}_{\Tr}$ defined by \[ [\pi]_{\Tr} \mapsto Z_{\Comp(\Tr(\pi))},\] is an algebra isomorphism between $\mathcal{A}^{\Pi}_{\Tr}$ and the subalgebra of $\QSym$ spanned by basis elements \[\{Z_{\delta} : \delta = (\delta_1, \ldots, \delta_k)\vDash n, \delta_1 > 1, n \in \N \}.\]
    \end{enumerate}
\end{proposition}

\begin{proof}
For $\tau, \theta \in \Pi$, let $S \subseteq [|\tau|+|\pi|]$ with $|S|=|\tau|$, and let $\pi = \shift_{S}(\tau) \cup \shift_{[|\tau|+|\theta|] \setminus S}(\theta)$.   Then $\Tr(\pi) = \shift_S(\Tr(\tau)) \cup \shift_{[|\tau|+|\theta|] \setminus S}(\Tr(\theta))$. Thus, the $\Tr$ value for the set of shifted arc-shuffles of $\tau$ and $\theta$ can be determiend only from $\Tr(\tau), \Tr(\theta), |\tau|, $ and $|\theta|$.

Observe that, for $\pi \in \Pi_n$, the elements $1$ and $n$ will never be transient as they must be an opener and a closer, respectively. The set of possible values for $\Tr(\pi)$ is given by the set of subsets of $[2,n-1]$. A given $T \subseteq [2,n-1]$ is achieved as the transient set of the set partition with a block $T \cup \{1,n\}$ and all other blocks singletons. Under the $\Comp$ bijection between subsets of $[n-1]$ and compositions of $n$, the possible transient sets correspond exactly to compositions whose first element is strictly greater than 1.
Let $\varsigma$ be the map $[\pi]_{\Tr} \mapsto Z_{\Comp(\Tr(\pi))}$, which is bijective by the logic above. Observe that 
\begin{align*}
\varsigma([\tau]_{\Tr}[\theta]_{\Tr}) &= \varsigma(\sum_{\pi \in \tau \smix \theta} [\pi]_{\Tr}) = \sum_{\pi \in \tau \smix \theta} Z_{\Comp(\Tr(\pi))} = \sum_{S \in [|\tau|+|\theta|], |S|=|\tau|} Z_{\Comp(\shift_S(\Tr(\tau)) \cup \shift_{S^c}(\Tr(\theta)))}\\
&= \sum_{v \in \epsilon(\Comp(\Tr(\tau)))\shuffle\epsilon(\Comp(\Tr(\theta))) } Z_{\epsilon^{-1}(v)} = Z_{\Comp(\Tr(\tau))}Z_{\Comp(\Tr(\theta))}=\varsigma([\tau]_{\Tr})\varsigma([\theta]_{\Tr}).
\end{align*}  
Thus, $\varsigma$ is a isomorphism from $\mathcal{A}^{\Pi}_{\Tr}$ to the specified subalgebra of $\QSym$.
\end{proof}

For $\pi \vdash [n]$, let $\Op^{rc}(\pi) = \{ n+1-i : i \not\in \Op(\pi) \}$.
\begin{proposition}\label{prop:opener} (Weak shuffle-compatibility of the opener set)
\begin{enumerate}
\item The set of openers statistic $\Op$ is weakly shuffle-compatible on set partitions.
\item The linear map on $\mathcal{A}_{\Op}^{\Pi}$ defined by \[[\pi]_{\Op} \rightarrow Z_{\Comp(\Op^{rc}(\pi))},\] is an algebra isomorphism between $\mathcal{A}_{\Op}^{\Pi}$ and $\QSym$.
\end{enumerate}
\end{proposition}

\begin{proof}
 For $\tau \in \Pi_n, \theta \in \Pi_m$, let $S \subseteq [n+m]$ with $|S|=n$, and let $\pi = \shift_{S}(\tau) \cup \shift_{[n+m] \setminus S}(\theta)$.   Then $\Op(\pi) = \shift_S(\Op(\tau)) \cup \shift_{[n+m] \setminus S}(\Op(\theta))$. Thus, the $\Op$ value for the set of shifted arc-shuffles of $\tau$ and $\theta$ can be determined only from $\Op(\tau), \Op(\theta), |\tau|, $ and $|\theta|$.
 
 Observe that the possible opener sets for $\pi \in \Pi_{n+m}$ are exactly subsets of $[n+m]$ that include $1$.  Given any such subset $S = \{s_1 < s_2 < \cdots < s_k  \} \subseteq [n+m]$ with $s_1=1$, we can create a set partition $\{1, \ldots, s_2-1\}/\{s_2, s_2+1, \ldots, s_3-1\}/ \cdots /\{s_k, s_k+1, \ldots, n+m\}$ that will have the opener set $S$. It follows that the possible values of $\Op^{rc}$ are exactly the subsets of $[n+m-1]$, which are in bijection with compositions of $n+m$. Thus, if $\varsigma$ denotes the map $[\pi]_{\Op} \rightarrow Z_{\Comp(\Op^{rc}(\pi))}$, then $\varsigma$ is bijective. Additionally, 
 \begin{align*}
 \varsigma([\tau]_{\Op}[\theta]_{\Op}) &= \varsigma(\sum_{\pi \in \tau \smix \theta }[\pi]_{\Op}) = \sum_{\pi \in \tau \smix \theta} Z_{\Comp(\Op^{rc}(\pi))}\\ &= \sum_{S \subseteq [n+m], |S|=n} Z_{\Comp(\shift_S(\Op^{rc}(\tau)) \cup \shift_{[n+m]\setminus S}(\Op^{rc}(\theta)))} \\ &= \sum_{v \in \epsilon(\Comp(\Op^{rc}(\tau))) \shuffle \epsilon(\Comp(\Op^{rc}(\theta))) } Z_{\epsilon^{-1}(v)} = Z_{\Comp(\Op^{rc}(\tau))}Z_{\Comp(\Op^{rc}(\theta))} = \varsigma([\tau]_{\Op})\varsigma([\theta]_{\Op}).
 \end{align*}
 Thus, $\varsigma$ is an isomorphism between $\mathcal{A}_{\Op}^{\Pi}$ and $\QSym$.
\end{proof}

For $\pi \vdash [n]$, let $\Cl^c(\pi) = \{ i : i \not\in \Cl(\pi) \}$.

\begin{proposition}\label{prop:cl} (Weak shuffle-compatibility of the closer set)
\begin{enumerate}
    \item The set of closers statistic $\Cl$ is weakly shuffle-compatible on set partitions.
    \item  The linear map on $\mathcal{A}_{\Cl}^{\Pi}$ defined by \[ [\pi]_{\Cl} \rightarrow Z_{\Comp(\Cl^c(\pi))},\] is an algebra isomorphism between $\mathcal{A}_{\Cl}^{\Pi}$ and $\QSym$.
\end{enumerate}
    
\end{proposition}

\begin{proof} The proof here follows the proof of Proposition~\ref{prop:opener} nearly exactly, substituting $\Cl$ in place of $\Op$ and $\Cl^c$ in place of $\Op^{rc}$. This works because the possible closer sets for $\pi \in \Pi_{n+m}$ are exactly the subsets of $[n+m]$ that include $n+m$. Given any such subset $S = \{ s_1 < s_2 < \cdots < s_k\} \subseteq [n+m]$ with $s_k = n+m$, we can create a set partition $\{1, \ldots, s_1\}/\{s_1+1, \ldots, s_2\}/ \cdots /\{s_{k-1}+1, \ldots, s_{k}\}$ that will have a closer set $S$.  It follows that the possible values of $\Cl^c$ are exactly subsets of $[n+m-1]$.
\end{proof}

\begin{corollary} The largest opener $\mathsf{lop}$ and smallest closer $\mathsf{scl}$ statistics are weakly shuffle-compatible on set partitions. 
\end{corollary}
\begin{proposition}\label{prop:bk1}(Weak shuffle-compatibility of the cardinality of the block containing $1$)
\begin{enumerate}
    \item The cardinality of the block containing 1 statistic $\bk_1$ is weakly shuffle-compatible on set partitions.
    \item The linear map on $\mathcal{A}_{\bk_1}^{\Pi}$ defined by \[[\pi]_{\bk_1}^{\Pi} \mapsto [f(\pi)]_{\lel}^{\PF},\] where $f: \Pi \rightarrow \PF$ sends a set partition $\pi$ to the parking function of the form  $(1)^{\bk_1(\pi)}\cdot (2)^{|\pi|-\bk_1(\pi)}$, is an algebra isomorphism between $\mathcal{A}^{\Pi}_{\bk_1}$ and $\mathcal{A}^{\PF}_{\lel}$.
\end{enumerate}
\end{proposition}
\begin{proof}
    Let $\tau \in \Pi_n$ and $\theta \in \Pi_m$. There are $\binom{n+m-1}{n-1}$ shuffles of $\tau$ and $\theta$ where $1$ stays in the part of the partition corresponding to $\tau$, and there are $\binom{n+m-1}{m-1}$ shuffles of $\tau$ and $\theta$ where $1$ stays in the part of the partition corresponding to $\theta$. Thus, in the set of all shuffles $\pi$ of $\tau$ and $\theta$ there will be $\binom{n+m-1}{n-1}$ shuffles  with $\bk_1(\pi)=\bk_1(\tau)$ and the remaining $\binom{n+m-1}{m-1}$ shuffles will have $\bk_1(\pi)=\bk_1(\theta)$. The map defined in the statement above sends the $\bk_1$-equivalence class with representative $\pi$ to the $\lel$-equivalence class of parking functions with representative $f(\pi)$ where we have defined $f$ so that $\bk_1(\pi)=\lel(f(\pi))$. 
    Comparing the work above with the proof of Proposition \ref{prop:lel} shows that the map is an isomorphism.
\end{proof}

\begin{proposition} 
The closer of the block containing 1 statistic $\mathsf{cl}_1$ is weakly shuffle-compatible on set partitions.
\end{proposition}
\begin{proof}
    For $\tau \in \Pi_n, \theta \in \Pi_m$, consider $\pi \in \tau \smix \theta$ where $\pi = \shift_S(\tau) \cup \shift_{S^c}(\theta)$ for some $S = \{s_1 <s_2 < \cdots < s_n\} \subseteq [n+m]$ and $S^c = [n+m]\setminus S =  \{r_1 < r_2 < \cdots < r_m\}$.  Then, \[\mathsf{cl}_1(\pi) = \begin{cases}
        s_{\mathsf{cl}_1(\tau)} & \text{ if } 1 \in S,\\
        r_{\mathsf{cl}_1(\theta)} & \text{ if }1 \in S^c.
    \end{cases}\]
    Thus, the set $\{\{ \mathsf{cl}_1(\pi) : \pi \in \tau \smix \theta \}\}$ depends only on $\mathsf{cl}_1(\tau)$, $\mathsf{cl}_1(\theta)$, $|\tau|$, and $|\theta|$.
\end{proof}

\subsection{Shuffle-compatibility for set partitions}

We now consider statistics that are shuffle-compatible on $\Psi$ and their algebraic structures.

\begin{definition} A statistic $\St$ is \emph{shuffle-compatible on set partitions} if 
\begin{enumerate}
    \item the multiset $\{\{\St(\pi) : \pi \in \tau \mix \theta\}\}$ is determined by $\St(\tau)$, $\St(\theta)$, $|\tau|$, and $|\theta|$ for any two disjoint set partitions $\tau, \theta \in \Psi$, and
    \item if $\std(\tau)=\std(\theta),$ then $\St(\tau)=\St(\theta)$ for any two set partitions $\tau, \theta \in \Psi$.
\end{enumerate}

\end{definition}

Equivalently, $\St$ is shuffle-compatible on set partitions if \[ \{\{ \St(\pi) : \pi \in \tau \mix \theta \}\} = \{\{ \St(\pi) : \rho \mix \nu \}\}\]
for any two pairs of disjoint set partitions $\{\tau, \theta\}$ and $\{ \rho, \nu \}$ such that $|\tau|=|\rho|$, $|\theta|=|\nu|$, $\St(\tau)=\St(\rho)$, and $\St(\theta) = \St(\nu)$. 

We denote the $\St$ equivalence class of $\pi \in \Psi$ by $[\pi]_{\St}^{\Psi}=\{ \tau \in \Psi : \St(\pi)=\St(\tau), |\pi|=|\tau|\}$, where we may drop the $\Psi$ superscript when it is clear from context. Let $\Psi/\St$ denote the set of all such $\St$ equivalence classes.

\begin{definition} Let $\St$ be shuffle-compatible on set partitions.
    The $\St$ shuffle algebra, denoted $\mathcal{A}^{\Psi}_{\St}$, is defined as the algebra with basis elements $[\pi]_{\St}^{\Psi}\in \Psi/\St$ and the product \[[\tau]_{\St}[\theta]_{\St} = \sum_{\pi \in \tau \mix \theta} [\pi]_{\St}.\]
\end{definition}

Shuffle algebras on set partitions relate closely to shifted shuffle algebras of set partitions and $\NCSym^*$.

\begin{proposition}\label{prop:alg_ncsym} Let $\St$ be a shuffle compatible statistic on set partitions.
\begin{enumerate}
    \item The statistic $\St$ is weakly shuffle-compatible on set partitions.
    \item The shuffle algebra $\mathcal{A}^{\Psi}_{\St}$ is isomorphic to the shifted shuffle algebra $\mathcal{A}^{\Pi}_{\St}$.
    \item The shuffle algebra $\mathcal{A}^{\Psi}_{\St}$ is a quotient of $\NCSym^*$.
\end{enumerate}
        
\end{proposition}

\begin{proof}
  Consider the map $\psi: \mathcal{A}^{\Psi}_{\St} \rightarrow \mathcal{A}^{\Pi}_{\St}$ defined by $[\pi]^{\Psi}_{\St} \mapsto [\std(\pi)]^{\Pi}_{\St}$.  By the definition of shuffle-compatibility on set partitions, $\St(\pi) = \St(\std(\pi))$ for any $\pi \in \Psi$. As a result, $[\std(\pi)]^{\Psi}_{\St}$ is exactly the subset of $[\pi]^{\Pi}_{\St}$ made up of set partitions
  in $\Pi$. Thus, $\psi$ is a bijection. Observe that \[\psi([\tau]_{\St}^{\Psi}[\theta]^{\Psi}_{\St}) = \psi(\sum_{\pi \in \tau \mix \theta} [\pi]^{\Psi}_{\St}) = \sum_{\pi \in \tau \mix \theta} [\std(\pi)]^{\Pi}_{\St}.\] On the other hand, \[\psi([\tau]^{\Psi}_{\St})\psi([\theta]^{\Psi}_{\St}) = [\std(\tau)]^{\Pi}_{\St}[\std(\theta)]_{\St}^{\Pi} = \sum_{\rho \in \std(\tau) \smix \std(\theta)} [\rho]^{\Pi}_{\St}.\] These two expressions are equivalent given that $\std(\tau) \smix \std(\theta) = \std(\tau) \mix \std(\theta)[|\tau|]$ and the shuffle-compatibility of $\St$ implies that $\{\{ \St(\pi) : \pi \in \tau \mix \theta \}\} = \{\{\St(\rho) : \rho \in \std(\tau) \mix \std(\theta)[|\tau|] \}\}$. Therefore, $\psi$ is an isomorphism. This proves part (1) of the theorem, which, in combination with Proposition \ref{prop:setparquo}, proves part (2).
\end{proof}

An integer partition of $n$, denoted $\lambda \vdash n$, is a sequence of positive integers $\lambda = (\lambda_1, \lambda_2, \ldots, \lambda_{\ell})$ such that $\lambda_1 \geq \lambda_2 \geq \cdots \geq \lambda_{\ell}$ and $\sum_{i=1}^{\ell} \lambda_i = n$.
For a multiset $S \subseteq \N$, let $\lambda(S)$ be the integer partition made up of the entries in $S$.

\begin{theorem}\label{thm:bk} (Shuffle-compatibility of the block sizes)
\begin{enumerate}
    \item The block sizes statistic $\Bk$ is shuffle-compatible on set partitions.
    \item The linear map on $\mathcal{A}^{\Psi}_{\Bk}$ defined by \[[\pi]_{\Bk} \rightarrow \frac{1}{|\pi|!} h_{\lambda(\Bk(\pi))},\] is an algebra isomorphism between $\mathcal{A}^{\Psi}_{\Bk}$ and $\Sym$.
    \end{enumerate}
\end{theorem}

\begin{proof}
Observe that if $\pi \in \tau \mix \theta$, then $\Bk(\pi) = \Bk(\tau) \cup \Bk(\theta)$. Shuffle-compatibility of $\Bk$ follows.

   Let $\varsigma$ be the map defined above, and let $\tau \in \Psi_n, \theta \in \Psi_m$. Observe that
   \begin{align*}
        \varsigma([\tau]_{\Bk}[\theta]_{\Bk}) &= \varsigma(\sum_{\pi \in \tau \mix \theta} [\pi]_{\Bk})= \sum_{\pi \in \tau \mix \theta} \frac{1}{(n+m)!}h_{\lambda(\Bk(\pi))} = \frac{1}{(n+m)!}\binom{n+m}{n} h_{\lambda(\Bk(\tau) \cup \Bk(\theta))} \\ &= (\frac{1}{n!}h_{\lambda(\Bk(\tau))})(\frac{1}{m!}h_{\lambda(\Bk(\theta))} )  = \varsigma([\tau]_{\Bk})\varsigma([\theta]_{\Bk}). 
   \end{align*}
  Additionally, the $\Bk$-equivalence classes of $\Psi$ correspond exactly to integer partitions, which index bases of $\Sym$, and so $\varsigma$ is an isomorphism.
\end{proof}

The results on the next few statistics follow in a straightforward manner from the fact that the block sizes of any shuffle of two set partitions are exactly the union of the block sizes of those partitions, with only slight differences in the algebraic formulation based on the possible values of the specific statistics.

\begin{corollary} (Shuffle-compatibility of the number of blocks)
\begin{enumerate}
    \item The block number statistic $\bk$ is shuffle-compatible on set partitions.
    \item The linear map on $\mathcal{A}^{\Psi}_{\bk}$ defined by $[\pi]_{\bk} \mapsto \frac{1}{|\pi|!} q^{\bk(\pi)}x$ is an algebra isomorphism from $\mathcal{A}_{\bk}^{\Psi}$ to the span of $\{1\}\cup\{
    q^jx^n\}_{1 \leq n, 1 \leq j \leq n}$, a subalgebra of $\mathbb{Q}[q,x]$. 
\end{enumerate}
\end{corollary}

\begin{corollary} (Shuffle-compatibility of the rank)
\begin{enumerate}
    \item The rank statistic $\rank$ is shuffle-compatible on set partitions.
    \item The linear map on $\mathcal{A}^{\Psi}_{\rank}$ defined by $[\pi]_{\rank} \mapsto \frac{1}{|\pi|!} q^{\rank(\pi)}x$ is an algebra isomorphism from $\mathcal{A}_{\rank}^{\Psi}$ to the span of $\{1\}\cup\{
    q^jx^n\}_{1 \leq n, 0 \leq j \leq n-1}$, a subalgebra of $\mathbb{Q}[q,x]$. 
\end{enumerate}
\end{corollary}

\begin{corollary} (Shuffle-compatibility of the number of singletons)
\begin{enumerate}
    \item The number of singletons statistic $\sg$ is shuffle-compatible on set partitions.
    \item The linear map on $\mathcal{A}^{\Psi}_{\sg}$ defined by $[\pi]_{\sg} \mapsto \frac{1}{|\pi|!} q^{\sg}x^n$ is an algebra isomorphism from $\mathcal{A}_{\sg}^{\Psi}$ to the span of $\{1\}\cup\{
    q^jx^n\}_{1 \leq n, 0 \leq j \leq n-2} \cup \{q^nx^n\}_{1\leq n}$, a subalgebra of $\mathbb{Q}[q,x]$. 
\end{enumerate}
\end{corollary}
\begin{corollary}
    (Shuffle-compatibility of the number of nonsingeltons)
    \begin{enumerate}
        \item The number of nonsingleton blocks statistic $\mathsf{nsg}$ is shuffle-compatible on set partitions. 
        \item The linear map on $\mathcal{A}^{\Psi}_{\mathsf{nsg}}$ defined by $[\pi]_{\mathsf{nsg}} \mapsto \frac{1}{|\pi|!} q^{\mathsf{nsg}(\pi)}x^n$ is an algebra isomorphism from $\mathcal{A}_{\mathsf{nsg}}^{\Psi}$ to the span of $\{1\}\cup\{
    q^jx^n\}_{1 \leq n, 0 \leq j \leq \lfloor \frac{n}{2}\rfloor}$, a subalgebra of $\mathbb{Q}[q,x]$.
    \end{enumerate}
\end{corollary}
\begin{corollary} \label{cor:maxBk}
(Shuffle-compatibility of the maximal and minimal block sizes)
\begin{enumerate}
\item The maximal block size statistic $\mathsf{maxBk}$ and minimal block size statistic $\mathsf{minBk}$ are shuffle-compatible on set partitions.
\item The linear map on $\mathcal{A}^{\Psi}_{\mathsf{maxBk}}$ defined by \[[\pi]^{\Psi}_{\mathsf{maxBk}} \mapsto [f(\pi)]^{\PF}_{\mathsf{maxDisp}},\] where $f(\pi) = (1)^{\mathsf{maxBk(\pi)}}\cdot(\mathsf{maxBk}(\pi)+1, \ldots, |\pi|-1, |\pi|) \in \PF_{|\pi|}$, is an algebra isomorphism from $\mathcal{A}^{\Psi}_{\mathsf{maxBk}}$ to $\mathcal{A}^{\PF}_{\mathsf{maxDisp}}$.
\end{enumerate}
\end{corollary}
\begin{proof}
    The maximal and minimal block sizes of any shuffle of two set partitions will simply be the maximum or minimum of the maximal and minimal block sizes of those two set partitions.
    In the same way, the maximal displacement of any shuffle of two parking functions is simply the maximum of the values between the two. Thus, in constructing a map $f$ that sends a $\mathsf{maxBk}$-equivalence class of set partitions with a representative $\pi$ to the $\mathsf{maxDisp}$-equivalence class of parking functions with representative $f(\pi)$ where $\mathsf{maxDisp}(f(\pi))=\mathsf{maxBk}(\pi)-1$, we construct an isomorphism. To see the bijectivity of this map, observe that the possible $\mathsf{maxBk}$ values of a set partition of size $n$ range from $1$ to $n$ while the $\mathsf{maxDisp}$ values range from $0$ to $n-1$.
\end{proof}

\begin{proposition} 
The number of terminal closers statistic $\mathsf{tcl}$ is shuffle-compatible on set partitions.
\end{proposition}
\begin{proof}
   Consider disjoint set partitions $\tau \vdash I$ with $|I|=n$ and $\theta \vdash J$ with $|J|=m$ where $\tcl(\tau)=k_{\tau}$ and $\tcl(\theta) = k_{\theta}$. Let $\pi$ be an arc-shuffle of $\tau$ and $\theta$ corresponding to the set $X = \{x_1 < \cdots < x_n\} \sqcup Y = \{y_1 < \cdots < y_m\} = I \sqcup J$ where $\pi = \shift_X(\tau) \cup \shift_Y(\theta)$. Let $x_0 =0$ and $y_0 = 0$. 
   Then, \[\tcl(\pi) = \begin{cases}
       |\{i \in I \cup J : i > x_{n-k_{\tau}})\}| & \text{ if }x_{n-k_{\tau}} \geq y_{m-k_{\theta}}\\
       |\{i \in I \cup J : i > y_{m-k_{\theta}}\}| & \text{ if }x_{n-k_{\tau}} < y_{m-k_{\theta}}
   \end{cases}.\]
It follows that we can determine $\{\{ \tcl(\pi) : \pi \in \tau \shuffle \theta \}\}$ completely from $\tcl(\tau), \tcl(\theta), |\tau|,$ and $|\theta|$.
\end{proof}
\begin{theorem} (Shuffle-compatibility of the number of occurrences of the pattern $12$)
\begin{enumerate}
    \item The number of occurrences of the pattern $12$ statistic $\occ_{12}$ is shuffle-compatible on set partitions.
    \item The linear map on $\mathcal{A}_{\occ_{12}}^{\Psi}$ defined by \[[\pi]_{\occ_{12}} \mapsto \frac{1}{|\pi|!}q^{\occ_{12}(\pi)}x^{|\pi|},\] is an algebra isomorphism from $\mathcal{A}_{\occ_{12}}^{\Psi}$ to the span of $\{1\} \cup \{q^{\binom{\lambda_1}{2}}q^{\binom{\lambda_2}{2}}\cdots q^{\binom{\lambda_{\ell}}{2}}x^n\}_{1 \leq n, \lambda = (\lambda_1, \ldots, \lambda_{\ell}) \vdash n, }$, a subalgebra of $\mathbb{Q}[q,x]$.
\end{enumerate}
\end{theorem}
\begin{proof}
    Since the relative order of the entries in $\tau$ and $\theta$ is perserved when shuffling $\tau, \theta \in \Psi$, and no blocks between the two are merged, the number of occurences of the pattern $12$ in any shuffle of $\tau$ and $\theta$ is $\occ_{12}(\tau)+\occ_{12}(\theta)$. Thus, the statistic is shuffle-compatible on set partitions.
    
    Let $\pi \in \Psi_n$ have blocks $B_1, B_2, \ldots, B_k$.  The number of instances of the pattern $12$ in $\pi$ is $\binom{|B_i|}{2}$ for each block $B_i$. Thus, the possible values of $\occ_{12}$ over all partitions in $\Psi_n$ are given by $\binom{\lambda_1}{2}+\binom{\lambda_2}{2} + \cdots + \binom{\lambda_{\ell}}{2}$ for all $\lambda \vdash n$.  This establishes that the proposed map is a bijection, and thus an isomorphism given its additive nature described above.
\end{proof}
\begin{theorem} (Shuffle-compatibility of the number of occurrences of the pattern $123$)
\begin{enumerate}
    \item The number of occurrences of the pattern $123$ statistic $\occ_{123}$ is shuffle-compatible on set partitions.
    \item The linear map on $\mathcal{A}_{\occ_{123}}^{\Psi}$ defined by \[[\pi]_{\occ_{123}} \mapsto \frac{1}{|\pi|!}q^{\occ_{123}(\pi)}x^{|\pi|},\] is an algebra isomorphism from $\mathcal{A}_{\occ_{123}}^{\Psi}$ to the span of $\{1\} \cup \{q^{\binom{\lambda_1}{3}}q^{\binom{\lambda_2}{3}}\cdots q^{\binom{\lambda_{\ell}}{3}}x^n\}_{1 \leq n, \lambda = (\lambda_1, \ldots, \lambda_{\ell}) \vdash n, }$, a subalgebra of $\mathbb{Q}[q,x]$.
\end{enumerate}
\end{theorem}
\begin{proof}
 Since the relative order of the entries in $\tau$ and $\theta$ is perserved when shuffling $\tau, \theta \in \Psi$, and no blocks between the two are merged, the number of occurences of the pattern $123$ in any shuffle of $\tau$ and $\theta$ is $\occ_{123}(\tau)+\occ_{123}(\theta)$. Thus, the statistic is shuffle-compatible on set partitions.
 
    Let $\pi \in \Psi_n$ have blocks $B_1, B_2, \ldots, B_k$.  The number of instances of the pattern $123$ in $\pi$ is $\binom{|B_i|}{3}$ for each block $B_i$. Thus, the possible values of $\occ_{123}$ over all partitions in $\Psi_n$ are given by $\binom{\lambda_1}{3}+\binom{\lambda_2}{3} + \cdots + \binom{\lambda_{\ell}}{3}$ for all $\lambda \vdash n$.  This establishes that the proposed map is a bijection, and thus an isomorphism given its additive nature described above.
\end{proof}
\begin{proposition}
The number of occurrences of the pattern $1/2$ statistic $\occ_{1/2}$ is shuffle-compatible on set partitions.
\end{proposition}
\begin{proof}
    Consider $\tau \in \Psi_n$ and $\theta \in \Psi_m$ with $\tau \vdash I$ and $\theta \vdash J$.
    Consider the shuffle $\pi \in \tau \shuffle \theta$ with $\pi = \shift_S(\tau) \cup \shift_{(I \cup J) \setminus S}(\theta)$ for a set $S \subseteq (I \cup J)$ with $|S|=n$.
    Given that the relative order of the elements in $\tau$ and $\theta$ is maintained, $\pi$ will have $\occ_{1/2}(\tau)+\occ_{1/2}(\theta)$ instances of the pattern $1/2$ where both elements are coming either from $\tau$ or from $\theta$. It remains to count the instances of $1/2$ where the elements are in separate blocks. Each pair of elements $(i,j)$ with $i\in S$ and $j \not \in S$ will be in separate blocks and thus be an instance of $1/2$. There are $nm$ such pairs, so total $\occ_{1/2}(\pi) = \occ_{1/2}(\tau)+\occ_{1/2}(\theta)+nm$.  Given this determination, $\occ_{1/2}$ is shuffle-compatible on set partitions.
\end{proof}

\appendix
\section{Non-shuffle-compatible statistics}\label{sec:appendix}

\subsection{Parking Functions}\label{ap:pf}  Table \ref{tab:ncpfstats} lists examples of statistics that are not weakly shuffle-compatible on parking functions. Note that these statistics cannot be shuffle-compatible on words (if defined).

\begin{table}[h]
    \centering
    \begin{tabular}{|c|c|c|} \hline
    \textbf{Name}  & \textbf{Reference} & \textbf{Counterexample} \\ \hline
    Peak set &   & $(1)\shifted (1,1,1)$ and $(1) \shifted (3,2,1)$  \\ \hline
    Number of peaks  &  St000023 & $(1)\shifted (1,1,1)$ and $(1) \shifted (3,2,1)$ \\ \hline
        Left peak set & & $(1)\shifted (1,1,1)$ and $(1) \shifted (1,2,3)$ \\ \hline
        Right peak set & & $(1)\shifted (1,1,1)$ and $(1) \shifted (3,2,1)$  \\ \hline
        Exterior peak set & & $(1)\shifted (1,1,1)$ and $(1) \shifted (1,2,2)$ \\ \hline
         Number of fixed points & St001903  & $(1)\shifted (1,3,2)$ and $(1) \shifted (3,2,1)$  \\ \hline
        Number of excedances & St000155 & $(1) \shifted (1,3,2)$ and $(1)\shifted (3,1,2)$ \\ \hline
   dinv & St000136 & $(1,1) \shifted(1,1)$ and $(1,1)\shifted(2,1)$ \\ \hline
        Primary dinv &   St000194 & $(1,1) \shifted (1)$ and $(2,1)\shifted(1)$ \\  \hline
        Secondary dinv &  	
St000195 & $(1,1) \shifted (1)$ and $(2,1)\shifted(1)$ \\ \hline
Number of critical L-to-R maxima &  St000942 & $(1,3,1) \shifted(1)$ and $(1,2,2)\shifted(1)$ \\ \hline
pmaj & St001209 & $(3,1,2) \shifted (1)$ and $(1,3,2) \shifted (1)$ \\ \hline
Number of preferences less than indices &  St001905 & $(1,1,3) \shifted (1)$ and $(1,2,2) \shifted (1)$ \\ \hline
   Size of center & St001937 & $(1) \shifted (1,3,2)$ and $(1)\shifted (3,1,2)$  \\ \hline
   Mak & St000794 & $(1) \shifted (2,1,3)$ and $(1) \shifted (3,1,2)$ \\ \hline 
   Mad & St000795 &  $(1) \shifted (1,3,2)$ and $(1) \shifted (2,1,3)$ \\ \hline
   Denert's index & St000156 & $(1) \shifted (2,1,3)$ and $(1) \shifted (3,1,2)$ \\ \hline
   Number of up-down runs & St000638 & $(1) \shifted (2,3,1)$ and $(1) \shifted (1,1,2)$ \\ \hline
    \end{tabular}
    \caption{Non-weakly-shuffle-compatible statistics on parking functions}
    
    \label{tab:ncpfstats}
\end{table}

\subsection{Set Partitions} Table \ref{tab:ncsp} lists examples of statistics that are not shuffle-compatible (although they may be weakly shuffle-compatible) on set partitions. Tables \ref{tab:nwcsp} and \ref{tab:ncpo} list examples of statistics that are not weakly shuffle-compatible on set partitions, and thus also not shuffle-compatible on set partitions.

\begin{table}[h!]
    \centering
    \begin{tabular}{|c|c|} \hline
    \textbf{Name}  & \textbf{Counterexample} \\ \hline
    Succession set & $12 \mix 3/4$ and $12 \mix 5/6$ \\ \hline
    Succession number    & $13 \mix 24$ and $15 \mix 37$ \\ \hline
    Set of openers   &  $1 \mix 23 $ and $1 \mix 24$  \\ \hline
        Set of closers   & $1 \mix 34$ and $1 \mix 24$ \\ \hline
        Set of singletons   & $1 \mix 23$ and $1 \mix 24$ \\ \hline
        Set of transients  & $1 \mix 245$ and $1 \mix 345$ \\ \hline   
        Biggest entry in the block containing 1 & $2 \mix 14/3$ and $2 \mix 14/5$ \\ \hline
    \end{tabular}
    \caption{Non-shuffle-compatible statistics on set partitions}
    \label{tab:ncsp}
\end{table}

\newpage
\begin{table}[h!]
    \centering
    \begin{tabular}{|c|c|c|c|} \hline
    \textbf{Name} & \textbf{FindStat} & \textbf{Counterexample} \\ \hline
        Size of orbit &  St000163 & $1 \smix 1/2$ and $1 \smix 12$  \\ \hline
        Dimension index & St000229  & $1 \smix 1/2$ and $1 \smix 12$ \\ \hline
       Sum of minimal elements & St000230 & $1 \smix 1/2/345$ and $1 \smix 1234/5$  \\ \hline
       Sum of maximal elements & St000231 & $1 \smix 12/3/456$  and $1 \smix 12345/6$\\ \hline
        Number of crossings & St000232 & $12 \smix 1/2$ and $12 \smix 12 $\\ \hline
        Number of nestings  & St000233 & $12 \smix 1/2$ and $12 \smix 12 $ \\ \hline
       Number of antisingletons & St000248 & $1 \smix 12/3$ and $1 \smix 13/2$ \\ \hline
       Number of singletons plus antisingletons & St000249 & $1 \smix 12/3$ and $1 \smix 13/2$ \\ \hline 
       Number of blocks plus antisingletons & St000250 & $1 \smix 12/3$ and $1 \smix 13/2$ \\ \hline
     Crossing number & St000253 & $1/2 \smix 12$ and $12 \smix 12$ \\ \hline
     Nesting number & St000254 & $1/2 \smix 12$ and $12 \smix 12$ \\ \hline
    Intertwining number & St000490 & $1 \smix 13/24$ and $1 \smix 12/3/4$\\ \hline
        ros (right opener smaller), inversion number & St000491 & $1 \smix 1/2$ and $1 \smix 12$ \\ \hline
        rob (right opener bigger) &  St000492 & $1 \smix 123/4$ and $ 1 \smix 14/2/3$ \\ \hline
          
        los (left opener smaller) &  St000493 & $1 \smix 13/2/4 $ and $ 1 \smix 1/234$ \\ \hline
        rcs (right closer smaller) & St000496  & $1 \smix 1/2$ and $1 \smix 12$ \\ \hline
        lcb (left closer bigger) & St000497  & $1 \smix 1/2$ and $1 \smix 12$ \\ \hline
        lcs (left closer smaller) & St000498 & $1 \smix 1234$ and $1 \smix 124/3$ \\ \hline
        rcb (right closer bigger) & St000499  & $1 \smix 123/4$ and $ 1 \smix 14/2/3$ \\ \hline
        Maximal difference in one block & St000503 & $1 \smix 13/24$ and $1 \smix 123/4$   \\ \hline 
           Sagan's major index  & St000565 & $1 \smix 1/2$ and $1 \smix 12$ \\ \hline
           Number of internal points & St000562 & $1 \smix 123$ and $1 \smix 13/2$ \\ \hline
        Number of overlapping pairs of blocks  & St000563 & $12 \smix 1/2$ and $12 \smix 12$ \\ \hline
        Dimension exponents & St000572 & $1 \smix 12$ and $1 \smix 1/2$ \\ \hline
        Number of blocks in first part of atomic decomposition & St000695 & $1 \smix 12$ and $1 \smix 1/2$ \\ \hline
        Dimension & St000728 & $12 \smix 12/34 $ and $12 \smix 13/2/4$\\ \hline
        Minimal arc length & St000729 & $ 1 \smix 123$ and $1 \smix 12/3$ \\ \hline
         Maximal arc length & St000730 & $ 1 \smix 123$ and $1 \smix 12/3$ \\ \hline
         Variant of major index & St000747 & $1 \smix 12$ and $1 \smix 1/2$ \\ \hline
         Major index of permutation obtained by flattening & St000748 & $1 \smix 12$ and $1 \smix 1/2$ \\ \hline
         Length of the longest partition in vacillating tableau & St000793 & $1 \smix 13/2$ and $1 \smix 1/2/3$ \\ \hline  
        Number of unsplittable factors & St000823 & $1 \smix 123$ and $1 \smix 1/23$ \\ \hline
        Number of topologically connected components & St000925 & $12 \smix 12/3$ and $12 \smix 13/2$  \\ \hline  
        Depth of label 1 in associated tree & St001051 & $1 \smix 123$ and $1 \smix 13/2$ \\ \hline
        Depth index  & St001094 &  $12 \smix 12/34 $ and $ 12 \smix 123/4$\\ \hline
        Number of blocks with odd minimum & St001151 & $1 \smix 12$ and $1 \smix 1/2$ \\ \hline
        Number of blocks with even minimum & St001153 & $1 \smix 123$ and $1 \smix 12/3$ \\ \hline
        Number of ascent tops with all smaller elements before & St001641 & $1 \smix 12$ and $1 \smix 1/2$  \\ \hline
        Excess length of longest path in associated graph & St001693 & $1 \smix 123$ and $1 \smix 12/3$ \\ \hline
        Interlacing number & St001781 & $1 \smix 12 $ and $1 \smix 1/2$ \\ \hline
        Min of smallest closer, second element in block with 1 & 	
St001784  & $12 \smix 1/2/3$ and $12 \smix 1/23$ \\ \hline
        Number of Mahonian exceedances  & St001839 & $1 \smix 1/2$ and $1 \smix 12$ \\ \hline
        Number of Mahonian descents  & St001840 & $1 \smix 1/2$ and $1 \smix 12$ \\ \hline
        Number of Mahonian inversions  & St001841 & $1 \smix 1/2$ and $1 \smix 12$ \\ \hline
        Mahonian major index & St001842 & $1 \smix 123$ and $1 \smix 12/3$ \\ \hline
        Z-index & St001843 & $ 1 \smix 1/2$ and $1 \smix 12$ \\ \hline 
     
    \end{tabular}
    \caption{Non-weakly-shuffle-compatible statistics on set partitions}
    \label{tab:nwcsp}
\end{table}

\begin{table}[H]
    \centering
    \begin{tabular}{|c|c|c|} \hline
       \textbf{Number of occurrences of the pattern...}  & \textbf{Findstat} & \textbf{Counterexample}\\ \hline
        12/3 & St000554 & $12 \smix 1/2$ and $1/2 \smix 1/2$ \\ \hline
         13/2 & St000555 & $12 \smix 1/2$ and $1/2 \smix 1/2$ \\ \hline
        1/23 & St000556 & $12 \smix 1/2$ and $1/2 \smix 1/2$ \\ \hline
        1/2/3 & St000557 & $12 \smix 1/2$ and $1/2 \smix 1/2$ \\ \hline
         13/24 & St000559 &   $12 \smix 12$ and $1/2 \smix 1/2$  \\ \hline
         12/34 & St000560 & $12 \smix 12$ and $1/2 \smix 1/2$ \\ \hline
         1/2 such that 1 singleton, 2 closer & St000573 & $1 \smix 12/3$ and $1 \smix 123$\\ \hline
         1/2 such that 1 opener, 2 closer & St000574 & $1 \smix 15/23/4$ and $1 \smix 145/2/3$  \\ \hline
        1/2 such that 1 closer, 2 singleton & St000575 & $1 \smix 1/2/34$ and $1 \smix 123/4$  \\ \hline
       1/2 such that 1 closer, 2 opener  & St000576 & $ 1 \smix 1234$ and $1 \smix 134/2$ \\ \hline
        1/2 such that 1 closer & St000577 & $1 \smix 1/234$ and $1 \smix 14/2/3$ \\ \hline
        1/2 such that 1 singleton & St000578 & $1 \smix 12/3/4$ and $1 \smix 124/3$\\ \hline
       1/2 such that 2 closer & St000579 & $1 \smix 1/2/34$ and $1 \smix 123/4$ \\ \hline
        1/2/3 such that 2 opener, 3 closer & St000580 &  $12 \smix 12$ and $1/2 \smix 1/2$ \\ \hline
        13/2 such that 1 opener, 2 closer & St000581 & $12 \smix 12$ and $1/2 \smix 1/2$ \\ \hline
        13/2 with 1 opener, 3 closer, (1,3) consecutive in block  & St000582 & $12 \smix 12$ and $1/2 \smix 1/2$ \\ \hline
       1/2/3 with 3 opener, 1,2 closers & St000583 & $12 \smix 12$ and $1/2 \smix 1/2$ \\ \hline
        1/2/3 with 1 opener, 3 closer & St000584 & $12 \smix 12$ and $1/2 \smix 1/2$ \\ \hline
        13/2 with 2 closer, (1,3) consecutive in block  & St000585 & $12 \smix 12$ and $1/2 \smix 1/2$ \\ \hline
        1/23 such that 2 opener & St000586 & $12 \smix 12$ and $1/2 \smix 1/2$ \\ \hline
        1/2/3 such that 1 opener& St000587 & $12 \smix 12$ and $1/2 \smix 1/2$ \\ \hline
        1/2/3 such that 1,3 openers, 2 closer & St000588 & $12 \smix 12$ and $1/2 \smix 1/2$ \\ \hline
        1/23 with 1 closer, (2,3) consecutive in block& St000589 & $12 \smix 12$ and $1/2 \smix 1/2$ \\ \hline
       1/23 with 2 opener, (2,3) consecutive in block  & St000590 & $12 \smix 12$ and $1/2 \smix 1/2$ \\ \hline
        1/2/3 such that 2 closer & St000591 & $12 \smix 12$ and $1/2 \smix 1/2$ \\ \hline
        1/2/3  such that 1 closer & St000592 & $12 \smix 12$ and $1/2 \smix 1/2$ \\ \hline
       1/2/3 such that 1,2 openers & St000593 & $12 \smix 12$ and $1/2 \smix 1/2$ \\ \hline
         13/2 with 1,2 openers, (1,3) consecutive in block & St000594 & $12 \smix 12$ and $1/2 \smix 1/2$ \\ \hline
       1/23 such that 1 opener & St000595 & $12 \smix 12$ and $1/2 \smix 1/2$ \\ \hline
        1/2/3 such that 3 opener, 1 closer& St000596 & $12 \smix 12$ and $1/2 \smix 1/2$ \\ \hline
        1/23 with 2 opener, (2,3) consecutive in block & St000597 & $12 \smix 12$ and $1/2 \smix 1/2$ \\ \hline
        1/23 with 1,2 openers, 3 closer, (2,3) consecutive in block& St000598 & $12 \smix 12$ and $1/2 \smix 1/2$ \\ \hline
        1/23 with (2,3) consecutive in block & St000599 & $12 \smix 12$ and $1/2 \smix 1/2$ \\ \hline
        13/2 with 1 opener, (1,3) consecutive in block & St000600 & $12 \smix 12$ and $1/2 \smix 1/2$ \\ \hline
        1/23 with 1,2 openers, (2,3) consecutive & St000601 & $12 \smix 12$ and $1/2 \smix 1/2$ \\ \hline
        13/2 such that 1 opener & St000602 & $12 \smix 12$ and $1/2 \smix 1/2$ \\ \hline
        1/2/3 such that 2,3 openers & St000603 & $12 \smix 12$ and $1/2 \smix 1/2$ \\ \hline
       1/2/3 such that 3 opener, 2 closer & St000604 & $12 \smix 12$ and $1/2 \smix 1/2$ \\ \hline
       1/23 with 3 closer, (2,3) consecutive in block  & St000605 & $12 \smix 12$ and $1/2 \smix 1/2$ \\ \hline
        1/23 with 1,3 closers, (2,3) consecutive in block & St000606 & $12 \smix 12$ and $1/2 \smix 1/2$ \\ \hline
        1/23 with 2 opener, 3 closer, (2,3) consecutive in block & St000607 & $12 \smix 12$ and $1/2 \smix 1/2$ \\ \hline
       1/2/3 such that 1,2 openers, 3 closer & St000608 & $12 \smix 12$ and $1/2 \smix 1/2$ \\ \hline
        1/23 such that 1,2 openers & St000609 & $12 \smix 12$ and $1/2 \smix 1/2$ \\ \hline
        13/2 such that 2 closer & St000610 & $12 \smix 12$ and $1/2 \smix 1/2$ \\ \hline
        1/23 such that 1 closer & St000611 & $12 \smix 12$ and $1/2 \smix 1/2$ \\ \hline
        1/23 with 1 opener, (2,3) consecutive in block & St000612 & $12 \smix 12$ and $1/2 \smix 1/2$ \\ \hline
        13/2 with 2 opener, 3 closer, (1,3) consecutive in block & St000613 & $12 \smix 12$ and $1/2 \smix 1/2$ \\ \hline
        1/23 with 1 opener, 3 closer, (2,3) consecutive in block & St000614 & $12 \smix 12$ and $1/2 \smix 1/2$ \\ \hline
        1/2/3 such that 1,3 closers & St000615 & $12 \smix 12$ and $1/2 \smix 1/2$ \\ \hline      
    \end{tabular}
    \caption{Non-weakly-shuffle-compatible pattern-occurrence statistics on set partitions}
    \label{tab:ncpo}
\end{table}

\newpage
\printbibliography

@article {Ges18,
     AUTHOR = {Gessel, Ira M. and Zhuang, Yan},
     TITLE = {Shuffle-compatible permutation statistics},
   JOURNAL = {Adv. Math.},
  FJOURNAL = {Advances in Mathematics},
    VOLUME = {332},
      YEAR = {2018},
     PAGES = {85--141},
      ISSN = {0001-8708,1090-2082},
   MRCLASS = {05A05 (05A15 05E05 16T30)},
  MRNUMBER = {3810249},
MRREVIEWER = {Sam\ Hopkins},
       DOI = {10.1016/j.aim.2018.05.003},
       URL = {https://doi.org/10.1016/j.aim.2018.05.003},
}

@article{Mona,
    AUTHOR = {B\'{o}na, Mikl\'{o}s},
     TITLE = {Real zeros and normal distribution for statistics on
              {S}tirling permutations defined by {G}essel and {S}tanley},
   JOURNAL = {SIAM J. Discrete Math.},
  FJOURNAL = {SIAM Journal on Discrete Mathematics},
    VOLUME = {23},
      YEAR = {2008/09},
    NUMBER = {1},
     PAGES = {401--406},
      ISSN = {0895-4801,1095-7146},
   MRCLASS = {05A05 (05A15 05A16)},
  MRNUMBER = {2476838},
MRREVIEWER = {Daniel\ E.\ Warren},
       DOI = {10.1137/070702254},
       URL = {https://doi.org/10.1137/070702254},
}

@Inbook{Luoto2013,
author="Luoto, Kurt
and Mykytiuk, Stefan
and van Willigenburg, Stephanie",
title="Hopf algebras",
bookTitle="An Introduction to Quasisymmetric Schur Functions: Hopf Algebras, Quasisymmetric Functions, and Young Composition Tableaux",
year="2013",
publisher="Springer New York",
address="New York, NY",
pages="19--50",
isbn="978-1-4614-7300-8",
doi="10.1007/978-1-4614-7300-8_3",
url="https://doi.org/10.1007/978-1-4614-7300-8_3"
}

@article{cruz2024some,
    AUTHOR = {Cruz, Ari and Harris, Pamela E. and Harry, Kimberly J. and
              Kretschmann, Jan and McClinton, Matt and Moon, Alex and
              Museus, John O.},
     TITLE = {On some discrete statistics of parking functions},
   JOURNAL = {J. Integer Seq.},
  FJOURNAL = {Journal of Integer Sequences},
    VOLUME = {27},
      YEAR = {2024},
    NUMBER = {8},
     PAGES = {Art. 24.8.6, 34},
      ISSN = {1530-7638},
      ISBN = {},
   MRCLASS = {05A05 (05A15)},
  MRNUMBER = {4850096},
}

@article{kasraoui2006distribution,
  AUTHOR = {Kasraoui, Anisse and Zeng, Jiang},
     TITLE = {Distribution of crossings, nestings and alignments of two
              edges in matchings and partitions},
   JOURNAL = {Electron. J. Combin.},
  FJOURNAL = {Electronic Journal of Combinatorics},
    VOLUME = {13},
      YEAR = {2006},
    NUMBER = {1},
     PAGES = {Research Paper 33, 12},
      ISSN = {1077-8926},
   MRCLASS = {05A18},
  MRNUMBER = {2212506},
MRREVIEWER = {Martin\ Klazar},
       DOI = {10.37236/1059},
       URL = {https://doi.org/10.37236/1059},
}

@article{novelli2016binary,
  AUTHOR = {Novelli, Jean-Christophe and Thibon, Jean-Yves},
     TITLE = {Binary shuffle bases for quasi-symmetric functions},
   JOURNAL = {Ramanujan J.},
  FJOURNAL = {Ramanujan Journal. An International Journal Devoted to the
              Areas of Mathematics Influenced by Ramanujan},
    VOLUME = {40},
      YEAR = {2016},
    NUMBER = {1},
     PAGES = {207--225},
      ISSN = {1382-4090,1572-9303},
   MRCLASS = {05E05 (11M32 16T30)},
  MRNUMBER = {3486001},
MRREVIEWER = {Eric\ S.\ Egge},
       DOI = {10.1007/s11139-016-9777-1},
       URL = {https://doi.org/10.1007/s11139-016-9777-1},
}

@article{liu2023shufflebasesquasisymmetricpower,
    author = {Ricky Ini Liu and Michael Tang},
    title = {Shuffle Bases and Quasisymmetric Power Sums},
    journal = {Combinatorial Theory},
    year = {2026},
    volume ={6},
    number = {1},
    DOI = {10.5070/C66165696}
}

@article{gessel2004refinement,
   AUTHOR = {Gessel, Ira M. and Seo, Seunghyun},
     TITLE = {A refinement of {C}ayley's formula for trees},
   JOURNAL = {Electron. J. Combin.},
  FJOURNAL = {Electronic Journal of Combinatorics},
    VOLUME = {11},
      YEAR = {2004/06},
    NUMBER = {2},
     PAGES = {Research Paper 27, 23},
      ISSN = {1077-8926},
   MRCLASS = {05A15 (05C05)},
  MRNUMBER = {2224940},
MRREVIEWER = {Pavlo\ Pylyavskyy},
       DOI = {10.37236/1884},
       URL = {https://doi.org/10.37236/1884},
}

@article{COLARIC2021102129,
   AUTHOR = {Colaric, Emma and DeMuse, Ryan and Martin, Jeremy L. and Yin,
              Mei},
     TITLE = {Interval parking functions},
   JOURNAL = {Adv. in Appl. Math.},
  FJOURNAL = {Advances in Applied Mathematics},
    VOLUME = {123},
      YEAR = {2021},
     PAGES = {Paper No. 102129, 17},
      ISSN = {0196-8858,1090-2074},
   MRCLASS = {05E16 (05A05 06D99 20F55)},
  MRNUMBER = {4175420},
MRREVIEWER = {Christian\ Stump},
       DOI = {10.1016/j.aam.2020.102129},
       URL = {https://doi.org/10.1016/j.aam.2020.102129},
}

@misc{FindStat,
author        = {Martin Rubey and Christian Stump and others},
title         = {{FindStat} - {T}he combinatorial statistics database},
howpublished  = {\url{http://www.FindStat.org}},
url           = {http://www.FindStat.org},
note          = {Accessed: \today},
}

@article{durmic2023probabilistic,
     AUTHOR = {Durmi\'{c}, Irfan and Han, Alex and Harris, Pamela E. and
              Ribeiro, Rodrigo and Yin, Mei},
     TITLE = {Probabilistic parking functions},
   JOURNAL = {Electron. J. Combin.},
  FJOURNAL = {Electronic Journal of Combinatorics},
    VOLUME = {30},
      YEAR = {2023},
    NUMBER = {3},
     PAGES = {Paper No. 3.18, 25},
      ISSN = {1077-8926},
   MRCLASS = {05A19 (05A16 60C05)},
  MRNUMBER = {4626349},
MRREVIEWER = {Martin\ V.\ Hildebrand},
       DOI = {10.37236/11649},
       URL = {https://doi.org/10.37236/11649},
}

@article{stanleyPF,
author = {Stanley, Richard and Yin, Mei},
year = {2026},
month = {02},
pages = {},
title = {Some enumerative properties of parking functions},
volume = {5},
journal = {Combinatorial Theory},
doi = {10.5070/C65465679}
}

@inproceedings{Nov04,
  AUTHOR = {Novelli, Jean-Christophe and Thibon, Jean-Yves},
  TITLE = {A Hopf algebra of parking functions},
  BOOKTITLE = {Proceedings of the $16^{th}$ International Conference on Formal Power Series and Algebraic Combinatorics},
  YEAR = {2004},
  SERIES = {FPSAC '04, Vancouver, Canada},
}

@incollection {Yan15,
    AUTHOR = {Yan, Catherine H.},
     TITLE = {Parking functions},
 BOOKTITLE = {Handbook of enumerative combinatorics},
    SERIES = {Discrete Math. Appl. (Boca Raton)},
     PAGES = {835--893},
 PUBLISHER = {CRC Press, Boca Raton, FL},
      YEAR = {2015},
      ISBN = {978-1-4822-2085-8},
   MRCLASS = {05A15 (05C85)},
  MRNUMBER = {3409354},
}

@article {Ros06,
    AUTHOR = {Rosas, Mercedes H. and Sagan, Bruce E.},
     TITLE = {Symmetric functions in noncommuting variables},
   JOURNAL = {Trans. Amer. Math. Soc.},
  FJOURNAL = {Transactions of the American Mathematical Society},
    VOLUME = {358},
      YEAR = {2006},
    NUMBER = {1},
     PAGES = {215--232},
      ISSN = {0002-9947,1088-6850},
   MRCLASS = {05E05 (05A18 05E10)},
  MRNUMBER = {2171230},
MRREVIEWER = {Peter\ R. W. McNamara},
       DOI = {10.1090/S0002-9947-04-03623-2},
       URL = {https://doi.org/10.1090/S0002-9947-04-03623-2},
}

@article {Ber08,
   AUTHOR = {Bergeron, Nantel and Reutenauer, Christophe and Rosas,
              Mercedes and Zabrocki, Mike},
     TITLE = {Invariants and coinvariants of the symmetric groups in
              noncommuting variables},
   JOURNAL = {Canad. J. Math.},
  FJOURNAL = {Canadian Journal of Mathematics. Journal Canadien de
              Math\'{e}matiques},
    VOLUME = {60},
      YEAR = {2008},
    NUMBER = {2},
     PAGES = {266--296},
      ISSN = {0008-414X,1496-4279},
   MRCLASS = {16W30 (05A18 05E10)},
  MRNUMBER = {2398749},
MRREVIEWER = {Oleg\ V.\ Ogievetsky},
       DOI = {10.4153/CJM-2008-013-4},
       URL = {https://doi.org/10.4153/CJM-2008-013-4},
}

@article{bakerjarvis2019bijectiveproofsshufflecompatibility,
      AUTHOR = {Baker-Jarvis, Duff and Sagan, Bruce E.},
     TITLE = {Bijective proofs of shuffle compatibility results},
   JOURNAL = {Adv. in Appl. Math.},
  FJOURNAL = {Advances in Applied Mathematics},
    VOLUME = {113},
      YEAR = {2020},
     PAGES = {101973, 29},
      ISSN = {0196-8858,1090-2074},
   MRCLASS = {05A05 (05A19)},
  MRNUMBER = {4032316},
MRREVIEWER = {David\ Callan},
       DOI = {10.1016/j.aam.2019.101973},
       URL = {https://doi.org/10.1016/j.aam.2019.101973},
}

@article{liang2023cyclicshufflecompatibilitycyclicshuffle,
       AUTHOR = {Liang, Jinting and Sagan, Bruce E. and Zhuang, Yan},
     TITLE = {Cyclic shuffle-compatibility via cyclic shuffle algebras},
   JOURNAL = {Ann. Comb.},
  FJOURNAL = {Annals of Combinatorics},
    VOLUME = {28},
      YEAR = {2024},
    NUMBER = {2},
     PAGES = {615--654},
      ISSN = {0218-0006,0219-3094},
   MRCLASS = {05A05 (05E05)},
  MRNUMBER = {4747489},
MRREVIEWER = {Brendon\ Rhoades},
       DOI = {10.1007/s00026-023-00669-9},
       URL = {https://doi.org/10.1007/s00026-023-00669-9},
}

@misc{carnevale2025colouredshufflecompatibilityhadamard,
      AUTHOR = {Carnevale, Angela and Moustakas, Vassilis Dionyssis and
              Rossmann, Tobias},
     TITLE = {Coloured shuffle compatibility, {H}adamard products, and ask
              zeta functions},
   JOURNAL = {Bull. Lond. Math. Soc.},
  FJOURNAL = {Bulletin of the London Mathematical Society},
    VOLUME = {57},
      YEAR = {2025},
    NUMBER = {7},
     PAGES = {2132--2154},
      ISSN = {0024-6093,1469-2120},
   MRCLASS = {05A15 (05E05 11M41 15B33 20D15 20E45)},
  MRNUMBER = {4937962},
MRREVIEWER = {Wayne\ M.\ Dymacek},
       DOI = {10.1112/blms.70081},
       URL = {https://doi.org/10.1112/blms.70081},
}

@article {yang2022conjectureconcerningshufflecompatiblepermutation,
    AUTHOR = {Yang, Lihong and Yan, Sherry H. F.},
     TITLE = {On a conjecture concerning shuffle-compatible permutation
              statistics},
   JOURNAL = {Electron. J. Combin.},
  FJOURNAL = {Electronic Journal of Combinatorics},
    VOLUME = {29},
      YEAR = {2022},
    NUMBER = {3},
     PAGES = {Paper No. 3.3, 11},
      ISSN = {1077-8926},
   MRCLASS = {05A15 (05A05)},
  MRNUMBER = {4446713},
MRREVIEWER = {Yan\ Zhuang},
       DOI = {10.37236/10953},
       URL = {https://doi.org/10.37236/10953},
}

@article{aliniaeifard2021hopfstructuresrepresentationtheory,
      AUTHOR = {Aliniaeifard, Farid and Thiem, Nathaniel},
     TITLE = {Hopf structures in the representation theory of direct
              products},
   JOURNAL = {Electron. J. Combin.},
  FJOURNAL = {Electronic Journal of Combinatorics},
    VOLUME = {29},
      YEAR = {2022},
    NUMBER = {4},
     PAGES = {Paper No. 4.39, 33},
      ISSN = {1077-8926},
   MRCLASS = {05E10 (16T30)},
  MRNUMBER = {4516937},
MRREVIEWER = {Lucio\ Centrone},
       DOI = {10.37236/11259},
       URL = {https://doi.org/10.37236/11259},
}

@article{ehrenborg1996posets,
  title={On posets and Hopf algebras},
  author={Ehrenborg, Richard},
  journal={advances in mathematics},
  volume={119},
  number={1},
  pages={1--25},
  year={1996},
  publisher={Elsevier}
}

@misc{adams2025cyclicsievingpermutations,
      title={Cyclic sieving on permutations -- an analysis of maps and statistics in the FindStat database}, 
      author={Ashleigh Adams and Jennifer Elder and Nadia Lafrenière and Erin McNicholas and Jessica Striker and Amanda Welch},
      year={2025},
      eprint={2402.16251},
      archivePrefix={arXiv},
      primaryClass={math.CO},
      url={https://arxiv.org/abs/2402.16251}, 
}

@article{Elder_2023,
     AUTHOR = {Elder, Jennifer and Lafreni\`ere, Nadia and McNicholas, Erin
              and Striker, Jessica and Welch, Amanda},
     TITLE = {Homomesies on permutations: an analysis of maps and statistics
              in the {F}ind{S}tat database},
   JOURNAL = {Math. Comp.},
  FJOURNAL = {Mathematics of Computation},
    VOLUME = {93},
      YEAR = {2024},
    NUMBER = {346},
     PAGES = {921--976},
      ISSN = {0025-5718,1088-6842},
   MRCLASS = {05E18},
  MRNUMBER = {4678590},
MRREVIEWER = {Konrad\ P.\ Pi\'{o}ro},
       DOI = {10.1090/mcom/3866},
       URL = {https://doi.org/10.1090/mcom/3866},
}

@article{stembridge1997enriched,
    AUTHOR = {Stembridge, John R.},
     TITLE = {Enriched {$P$}-partitions},
   JOURNAL = {Trans. Amer. Math. Soc.},
  FJOURNAL = {Transactions of the American Mathematical Society},
    VOLUME = {349},
      YEAR = {1997},
    NUMBER = {2},
     PAGES = {763--788},
      ISSN = {0002-9947,1088-6850},
   MRCLASS = {06A07 (05E05 05E10)},
  MRNUMBER = {1389788},
MRREVIEWER = {Joseph\ Neggers},
       DOI = {10.1090/S0002-9947-97-01804-7},
       URL = {https://doi.org/10.1090/S0002-9947-97-01804-7},
}

@incollection {mason2018recenttrendsquasisymmetricfunctions,
    AUTHOR = {Mason, Sarah K.},
     TITLE = {Recent trends in quasisymmetric functions},
 BOOKTITLE = {Recent trends in algebraic combinatorics},
    SERIES = {Assoc. Women Math. Ser.},
    VOLUME = {16},
     PAGES = {239--279},
 PUBLISHER = {Springer, Cham},
      YEAR = {2019},
      ISBN = {978-3-030-05141-9; 978-3-030-05140-2},
   MRCLASS = {05E05 (05-02 05E10)},
  MRNUMBER = {3969576},
       DOI = {10.1007/978-3-030-05141-9\{_}7}

@article{Kantarc_O_uz_2022,
    AUTHOR = {O\u{g}uz, Ezgi Kantarc\i },
     TITLE = {A counterexample to the shuffle compatiblity conjecture},
   JOURNAL = {Electron. J. Combin.},
  FJOURNAL = {Electronic Journal of Combinatorics},
    VOLUME = {29},
      YEAR = {2022},
    NUMBER = {3},
     PAGES = {Paper No. 3.51, 5},
      ISSN = {1077-8926},
   MRCLASS = {05A05},
  MRNUMBER = {4477850},
MRREVIEWER = {Yan\ Zhuang},
       DOI = {10.37236/10957},
       URL = {https://doi.org/10.37236/10957},
}

@article{Grinberg_2018,
 AUTHOR = {Grinberg, Darij},
     TITLE = {Shuffle-compatible permutation statistics {II}: the exterior
              peak set},
   JOURNAL = {Electron. J. Combin.},
  FJOURNAL = {Electronic Journal of Combinatorics},
    VOLUME = {25},
      YEAR = {2018},
    NUMBER = {4},
     PAGES = {Paper No. 4.17, 61},
      ISSN = {1077-8926},
   MRCLASS = {05E05 (05A05 06A11)},
  MRNUMBER = {3874283},
MRREVIEWER = {Sam\ Hopkins},
       DOI = {10.37236/7946},
       URL = {https://doi.org/10.37236/7946},
}

@incollection {vong2013algebraic,
    AUTHOR = {Vong, Vincent},
     TITLE = {Algebraic properties for some permutation statistics},
 BOOKTITLE = {25th {I}nternational {C}onference on {F}ormal {P}ower {S}eries
              and {A}lgebraic {C}ombinatorics ({FPSAC} 2013)},
    SERIES = {Discrete Math. Theor. Comput. Sci. Proc., AS},
     PAGES = {813--824},
 PUBLISHER = {Assoc. Discrete Math. Theor. Comput. Sci., Nancy},
      YEAR = {2013},
   MRCLASS = {05E05},
  MRNUMBER = {3091043},
}

@misc{grinberg,
        doi = {10.48550/ARXIV.1409.8356},
  
        url = {https://arxiv.org/abs/1409.8356},
  
        author = {Grinberg, Darij and Reiner, Victor},
  
        title = {Hopf Algebras in Combinatorics},
  
        publisher = {arXiv},
  
        year = {2014},
  
        copyright = {Creative Commons Attribution 4.0 International}
    }

@inproceedings{novelli2006polynomialrealizationstrialgebras,
  AUTHOR = {Novelli, Jean-Christophe and Thibon, Jean-Yves},
  TITLE = {Polynomial realizations of some trialgebras},
  BOOKTITLE = {Proceedings of the $18^{th}$ International Conference on Formal Power Series and Algebraic Combinatorics},
  YEAR = {2006},
  SERIES = {FPSAC '06, San Diego, California},
}

@phdthesis{hivert199,
    author = {Florent Hivert},
    title = {Combinatoire des fonctions quasi-symétriques},
    year = {1999}
}

@article{BERGERON_2009,
   AUTHOR = {Bergeron, Nantel and Zabrocki, Mike},
     TITLE = {The {H}opf algebras of symmetric functions and quasi-symmetric
              functions in non-commutative variables are free and co-free},
   JOURNAL = {J. Algebra Appl.},
  FJOURNAL = {Journal of Algebra and its Applications},
    VOLUME = {8},
      YEAR = {2009},
    NUMBER = {4},
     PAGES = {581--600},
      ISSN = {0219-4988,1793-6829},
   MRCLASS = {05E05 (05A18 05E10 16T30)},
  MRNUMBER = {2555523},
MRREVIEWER = {Aaron\ Lauve},
       DOI = {10.1142/S0219498809003485},
       URL = {https://doi.org/10.1142/S0219498809003485},
}

@article{Chapoton2000,
   AUTHOR = {Chapoton, Fr\'{e}d\'{e}ric},
     TITLE = {Alg\`ebres de {H}opf des permutah\`edres, associah\`edres et
              hypercubes},
   JOURNAL = {Adv. Math.},
  FJOURNAL = {Advances in Mathematics},
    VOLUME = {150},
      YEAR = {2000},
    NUMBER = {2},
     PAGES = {264--275},
      ISSN = {0001-8708,1090-2082},
   MRCLASS = {16W30 (05E99)},
  MRNUMBER = {1749253},
MRREVIEWER = {E.\ J.\ Taft},
       DOI = {10.1006/aima.1999.1868},
       URL = {https://doi.org/10.1006/aima.1999.1868},
}

@book {stanley1972ordered,
    AUTHOR = {Stanley, Richard P.},
     TITLE = {Ordered structures and partitions},
    SERIES = {Memoirs of the American Mathematical Society, No. 119},
 PUBLISHER = {American Mathematical Society, Providence, RI},
      YEAR = {1972},
     PAGES = {iii+104},
   MRCLASS = {05A17},
  MRNUMBER = {332509},
MRREVIEWER = {L.\ K.\ Durst},
}

\end{document}